\documentclass[11pt,reqno]{amsart}
\usepackage[margin=3cm]{geometry}
\usepackage{color}               
\usepackage{faktor}
\usepackage{tikz}
\usepackage{tikz-cd}
\usepackage{spverbatim}
\usepackage{float}
\usepackage{bbm}
\usepackage[pdftex,hyperfootnotes=false,colorlinks=false,hypertexnames=false]{hyperref}

\usepackage{amsmath,amssymb,anyfontsize,enumerate,stmaryrd,mathtools, thmtools,tikz,txfonts,nccmath}
\usepackage[oldstyle]{libertine}%
\usepackage[T1]{fontenc}
\usepackage[final]{microtype}
\usepackage[utf8]{inputenc}
\usepackage{csquotes}
\makeatletter
\renewcommand*\libertine@figurestyle{LF}
\makeatother
\usepackage[libertine,libaltvw,liby]{newtxmath}
\usepackage{multicol}
\makeatletter
\renewcommand*\libertine@figurestyle{OsF}
\makeatother
\usepackage[noerroretextools,backend=biber,style=alphabetic,maxalphanames=5,giveninits,sorting=nyt,%
maxbibnames=99]{biblatex}
\usepackage{tikz-cd}
\usepackage[noabbrev,capitalise]{cleveref}
\usepackage{autonum}
\usepackage{braket}
\theoremstyle{plain}
    \newtheorem{theorem}{Theorem}[section]
    \newtheorem{construction/theorem}[theorem]{Construction/Theorem}
    \newtheorem{corollary}[theorem]{Corollary}
    
    \newtheorem{proposition}[theorem]{Proposition}
    
\theoremstyle{definition}
    \newtheorem{remark}[theorem]{Remark}
    
    \newtheorem{example}[theorem]{Example}
    \newtheorem{definition}[theorem]{Definition}

\newcommand{\comment}[1]{}

\usepackage{caption}
\title{On the piecewise quasipolynomiality of double tropical Welschinger invariants}
\author[V.~Reda]{Vincenzo Reda}
\address{V.~Reda: School of Mathematics, 17 Westland Row, Trinity College Dublin, Dublin 2, Ireland}
\email{redav@tcd.ie}


\begin{document}

\begin{abstract}
    In \cite{ardila2017double}, Ardila and Brugallé conjectured that \textit{double tropical Welschinger invariants} of Hirzebruch surfaces are piecewise quasipolynomial. In this work, we prove the conjecture holds in a more general setting, i.e. for toric surfaces corresponding to $h$-transverse polygons. Furthermore, we define new combinatorial Welschinger-type numbers for $h$-transverse polygons and show that they are likewise piecewise quasipolynomial.
\end{abstract}


\maketitle

\tableofcontents

\section{Introduction}\label{sec:introduction}

Let $P$ be a convex polygon and $\Sigma$ be the toric surface associated to $P$. Let $g$ be a nonnegative integer and $\omega$ a configuration of $|\partial P\cap\mathbb Z^2|+g-1$ points in $\Sigma$. The number of curves defined over $\mathbb C$ that lie in $\Sigma$ and have genus $g$, fixed degree and pass through the points in $\omega$ is finite. Moreover, this number does not depend on $\omega$ as long as it is a generic configuration. The situation changes when we count real curves in $\Sigma$. Indeed, in this case the result heavily depends on $\omega$. However, in \cite{welschinger2003invariants,welschinger2005invariants}, Welschinger introduced real invariants, nowadays called \textit{Welschinger invariants}. Under suitable hypothesis on the surface $\Sigma$, these invariants are defined as weighted counts of real curves of genus $0$ and fixed degree in $\Sigma$ passing through a conjugation-invariant configuration of points, and depend only on the number of real points in the configuration. The weight of the curve is $\pm1$ and it is determined by the number of \textit{solitary nodes} of the curve, which are singularities that are locally given by the equation $x^2+y^2=0$.

In groundbreaking work \cite{mikhalkin2005enumerative}, Mikhalkin uses tropical geometry to count real curves of any genus $g\geq0$ in toric surfaces. Later on, in \cite{itenberg2009caporaso}, the authors define new tropical Welschinger numbers for relative constraints and any genus. They prove that these numbers are actually \textit{invariants under deformation} that coincide with the classical Welschinger invariants in the case in which the set of relative constraints is empty and the surface is del Pezzo. The result is surprising since these numbers are invariants only in the tropical world and no lift to the algebraic world is known. This work provides the motivation for the problem we introduce in the following sections.

For our purposes, we consider real plane parametrized tropical curves with marked points, see \cite[Definition 4.11]{brugalle2008floor}. These are tropical curves endowed with an involution and a parametrization such that their real part is defined as the set of points fixed by the involution and the complex part is the complement of the real part. The involution is also allowed to permute the marked points. In this paper, we refer to real plane parametrized tropical curves simply as real tropical curves and consider additional tangency conditions, as will be made precise later.

\subsection{Double tropical Welschinger invariants}\label{subsec:double_invariants}

Inspired by the suggestions outlined in \cite[Section 7]{ardila2017double} and following \cite{itenberg2009caporaso}, this section introduces the enumerative problem that forms the central subject of this paper. To this end, recall the notion of \textit{$h$-transverse polygons}.
\begin{definition}\label{def-h-transverse_polygons}
    Let $P\subset\mathbb R^2$ be a convex polygon. We say that $P$ is an \textbf{$h$-transverse polygon} if each side of $P$ has slope $0,\infty$ or $\frac1k$, where $k\in\mathbb Z$.
\end{definition}
Here, we consider only $h$-transverse polygons that have two nondegenerate parallel sides, i.e. sides that have non-zero lattice length. We refer to these sides as top and bottom side. A feature of these polygons is that they come with a parametrisation. More precisely, if $P$ is an $h$-transverse polygon, we can store the inverse of the slope of its right sides in the vector $\textbf{c}^r=(c_1^r,\dots,c_n^r)\in\mathbb Z^n$ such that $c_1^r>\dots>c_n^r$. Similarly, we store the inverse of the slope of its left sides in the vector $\textbf{c}^l=(c_1^l,\dots,c_m^l)\in\mathbb Z^m$ such that $c_1^l<\dots<c_m^l$. Moreover, let $d^t,d^b$ the lattice lengths of the top and bottom sides respectively and $d_i^r,d_j^l$ the lattice lengths of the right side corresponding to $c_i^r$ and the left side corresponding to $c_j^l$. We can completely determine $P$ using the vectors $\textbf{c}=(\textbf{c}^r;\textbf{c}^l)$ and $\textbf{d}=(d^t;\textbf{d}^r;\textbf{d}^l)$, where $\textbf{d}^r=(d_1^r,\dots,d_n^r)\in(\mathbb Z_{>0})^n$ and $\textbf{d}^l=(d_1^l,\dots,d_m^l)\in(\mathbb Z_{>0})^m$. We write $P=P(\textbf{c},\textbf{d})$. Note that $d^b$ is completely determined by $(\textbf{c},\textbf{d})$ since the normal fan to $P$ is balanced. In particular, the height of the polygon equals $\sum_{i=1}^nd_i^r=\sum_{j=1}^md_j^l$ and is denoted by $a$.

The next step is to extend the plane $\mathbb R^2$. Let $\mathbb T=\mathbb R\cup\{-\infty\}$ endowed with the topology that makes it homeomorphic to $[0,+\infty)$ via the logarithmic function and the binary operations
\begin{equation}
    x\oplus y=\max\{x,y\},\qquad x\odot y=x+y.
\end{equation}
Note that $0_{\mathbb T}=-\infty$, $1_{\mathbb T}=0$ and $x\odot y^{-1}$ is the usual difference $x-y$, for $y\neq-\infty$. $\mathbb T$ is called the \textbf{tropical semifield}.
\begin{definition}\label{def-tropical_projective_space}
    The \textbf{tropical projective space} $\mathbb T\mathbb P^n$ is the quotient space $(\mathbb T^{n+1}\setminus\{(-\infty,\dots,-\infty)\})/\sim$, where $(x_0,\dots,x_n)\sim (y_0,\dots,y_n)$ if and only if there exists $\lambda\in\mathbb T\setminus\{-\infty\}$ such that $x_i=y_i+\lambda$ for all $i=0,\dots,n$. As usual, we can define affine maps $\phi_i:\mathbb T\mathbb P^n\setminus\{x_i\neq-\infty\}\to\mathbb T^n$ by taking the tropical division by $x_i$. 
\end{definition}
Consider the extended plane $\hat{\mathbb R}^2=\mathbb R\times\mathbb T\mathbb P^1$ and correspondingly extend any tropical curve by attaching on $\mathsf B=\mathbb R\times\{-\infty\}\subset\mathbb R\times\phi_0(\mathbb T\mathbb P^1\setminus\{x_0\neq-\infty\})$ and $\mathsf T=\mathbb R\times\{-\infty\}\subset\mathbb R\times\phi_1(\mathbb T\mathbb P^1\setminus\{x_1\neq-\infty\})$ a vertex to each vertically directed end of the curve.\\
Let us denote by $\sigma_t$ and $\sigma_b$ the top and bottom sides of the $h$-transverse polygon $P(\textbf{c},\textbf{d})$. Let $\alpha,\beta,\tilde\alpha,\tilde\beta$ be four sequences of nonnegative integer numbers such that only finitely many terms are non-zero and
\begin{equation}
    \sum_ii(\alpha_i+\beta_i)=|\sigma_b\cap\mathbb Z^2|-1,\qquad\sum_ii(\tilde\alpha_i+\tilde\beta_i)=|\sigma_t\cap\mathbb Z^2|-1.
\end{equation}
Finally, let $l=2a+g+\sum_{i\ge1}(\beta_i+\tilde{\beta}_i)-1$.
\begin{definition}\label{def-DTW}
In the set-up above, consider a tropically generic configuration (in the sense of \cite[Section 4]{itenberg2009caporaso})
\begin{equation}
    \omega=\bigcup_{i\geq0}\left((q_j^i)_{1\leq j\leq\alpha_i}\cup(\tilde q_j^i)_{1\leq j\leq\tilde\alpha_i}\right)\cup(p_j)_{1\leq j\leq l},
\end{equation}
where $q_j^i\in \mathsf B,\tilde{q}_j^i\in \mathsf T$ and $p_j\in\mathbb R^2$. $W_{\textbf c,g}^{\alpha,\beta,\tilde\alpha,\tilde\beta}(\textbf{d})$ denotes the number of nodal irreducible real tropical curves $T$ counted with the multiplicity defined in \cite[Section 3]{itenberg2009caporaso} such that
    \begin{itemize}
        \item $T$ has Newton polygon $P(\textbf{c},\textbf{d})$ and is of genus $g$;
        \item all the non-vertical ends have weight $1$;
        \item the number of negatively directed vertical ends is equal to $\sum(\alpha_i+\beta_i)$ and the number of positively directed vertical ends is equal to $\sum(\tilde\alpha_i+\tilde\beta_i)$;
        \item $T$ passes through all the points in $\omega$, in particular the points $q_j^i$ are contained in a negatively directed vertical end of weight $i$ and the points $\tilde q_j^i$ are contained in a positively directed vertical end of weight $i$.
    \end{itemize}
\end{definition}
This number is finite and does not depend on the chosen generic configuration of points \cite{itenberg2009caporaso}. We call this number a \textbf{double tropical Welschinger invariant} of $P(\textbf{c},\textbf{d})$. When $\alpha=\beta=\tilde\alpha=\tilde\beta=0$, $g=0$ and the surface is del Pezzo we recover the classical Welschinger invariants.\\
In \cite{ardila2017double}, the authors propose that double tropical Welschinger invariants of Hirzebruch surfaces are expected to be piecewise quasipolynomial relative to the chambers of a certain hyperplane arrangement. This paper aims to answer this question affirmatively, not only in the case of Hirzebruch surfaces, but also for the wider family of $h$-transverse polygons. Additionally, we define new numbers that coincide with double tropical Welschinger invariants in certain circumstances and we prove that they are piecewise quasipolynomial too.

\subsection{Techniques}\label{subsec:techniques}

The task of counting tropical curves is often replaced by counting \textit{floor diagrams}. Floor diagrams are decorated graphs obtained from a degeneration process of a tropical curve. They are introduced in \cite{brugalle2007enumeration,brugalle2008floor,brugalle2016floor} and extensively used to solve enumerative problems. In \cite{ardila2013universal}, floor diagrams are employed to prove polynomiality of Severi degrees. In \cite{ardila2017double} and \cite{hahn2024universal} they play a central role in the study of piecewise polynomial behavior of double Gromov--Witten invariants.\\
In \cite{arroyo2011recursive}, a definition of real floor diagrams is provided to deal with relative tropical Welschinger invariants and prove, in a purely combinatorial way, the Caporaso--Harris type formula presented in \cite{itenberg2009caporaso}. In this work, we use the real version of these graphs suitably adapted to our case. Then, we prove \cref{thm-correspondence} which ensures that the floor diagrams count equals the tropical curves count.\\
The key idea to study the piecewise quasipolynomiality is to interpret the combinatorial problem in terms of weighted enumeration of lattice points in flow polytopes, and use non-trivial techniques from Ehrhart Theory. The main difficulty we face is that the lattice points in flow polytopes are weighted by an \textit{Ehrhart quasipolynomial} (see \cref{def-Ehrhartqp}) instead of a polynomial, which makes the problem harder than expected. In order to solve this issue, we use a theorem proved in \cite{de2024sums} to make the weighted sum of lattice points of flow polytopes an unweighted sum of lattice points of higher dimensional polytopes.

\subsection{Results}\label{subsec:results}

The primary objective of this paper is to prove the quasipolynomial behavior of the invariants $W_{\textbf c,g}^{\alpha,\beta,\tilde\alpha,\tilde\beta}(\textbf{d})$ defined in \cref{subsec:double_invariants}. To this end, we aim to construct a map $G_{(\textbf{d}^r;\textbf{d}^l),\textbf{c},g}^{n_1,n_2}(\textbf{x},\textbf{y})$ that encodes these numbers. Therefore, we define the set
\begin{equation}
    \Lambda=\left\{(x_1,\dots,x_{n_1},y_1,\dots,y_{n_2})\in\mathbb{Z}^{n_1}\times\mathbb{Z}^{n_2}\bigg|\sum_{i=1}^{n_1}x_i+\sum_{j=1}^{n_2}y_j+\sum_{i=1}^nc_i^rd_i^r-\sum_{j=1}^mc_j^ld_j^l=0\right\}
\end{equation}
where $n_1,n_2\geq0$ are integers. We associate a vector of sequences $(\alpha,\beta,\tilde\alpha,\tilde\beta)$ to $(\textbf{x},\textbf{y})\in\Lambda$ in the following way:
\begin{multicols}{2}
    \begin{itemize}
        \item $\alpha_i$ is the number of entries $x_j=-i$.
        \item $\beta_i$ is the number of entries $y_j=-i$.
        \item $\tilde\alpha_i$ is the number of entries $x_j=i$.
        \item $\tilde\beta_i$ is the number of entries $y_j=i$.
    \end{itemize}
\end{multicols}
Note that $(\alpha,\beta,\tilde\alpha,\tilde\beta)$ and $(\textbf{x},\textbf{y})$ determine each other up to permutation of the entries of $(\textbf{x},\textbf{y})$. Let us fix $g\in\mathbb Z_{\geq0}$, $\textbf{c}=(\textbf{c}^r;\textbf{c}^l)\in\mathbb Z^{n+m}$ such that $c_1^r>\dots>c_n^r$ and $c_1^l<\dots<c_m^l$ and $(\textbf{d}^r;\textbf{d}^l)\in(\mathbb Z_{>0})^{n+m}$, then we define the map
\begin{equation}
    G_{(\textbf{d}^r;\textbf{d}^l),\textbf{c},g}^{n_1,n_2}:\Lambda\longrightarrow\mathbb Z\qquad G_{(\textbf{d}^r;\textbf{d}^l),\textbf{c},g}^{n_1,n_2}(\textbf{x},\textbf{y})=W_{\textbf{c},g}^{\alpha,\beta,\tilde\alpha,\tilde\beta}(\textbf{d}).
\end{equation}
Moreover, if we let the vector $\textbf{c}$ vary, we can define the following map
\begin{equation}
    G_{(\textbf{d}^r;\textbf{d}^l),g}^{n_1,n_2}:\Lambda\times\mathbb Z^{n+m}\longrightarrow\mathbb Z\qquad G_{(\textbf{d}^r;\textbf{d}^l),g}^{n_1,n_2}(\textbf{x},\textbf{y},\textbf{c})=W_{\textbf{c},g}^{\alpha,\beta,\tilde\alpha,\tilde\beta}(\textbf{d}).
\end{equation}
Consider the hyperplane arrangement in $\Lambda$ defined by the following equations
\begin{equation}\label{eq-hyperplanes1}
    \sum_{i\in S}x_i+\sum_{j\in T}y_j+\sum_{i=1}^nk_ic_i^r-\sum_{j=1}^mt_jc_j^l=0
\end{equation}
\begin{equation}\label{eq-hyperplanes2}
    y_i-y_j=0\qquad 1\leq i<j\leq n_2 
\end{equation}
where $S\subseteq[n_1]$, $T\subseteq[n_2]$, $0\leq k_i\leq d_i^r$ and $0\leq t_j\leq d_j^l$ for all $i=1,\dots,n$ and $j=1,\dots,m$ and call it $\mathcal H^{n_1,n_2}(\textbf{c})$. Furthermore, we define $\tilde{\mathcal{H}}^{n_1,n_2}\subset\Lambda\times\mathbb Z^{n+m}$ to be the intersection of $\mathcal{H}^{n_1,n_2}(\textbf{c})\subset\Lambda\times\mathbb Z^{n+m}$ and $\{\textbf{c}=(\textbf{c}^r;\textbf{c}^l)\in\mathbb Z^{n+m}|c_1^r>\dots>c_n^r,\>c_1^l<\dots<c_m^l\}$. $\tilde{\mathcal{H}}^{n_1,n_2}$ is a hyperplane arrangement in $\Lambda\times\mathbb Z^{n+m}$. The following is the main result of this paper.
\begin{theorem}\label{thm-quasipolynomiality}
    Let $(\textbf{d}^r;\textbf{d}^l)\in\mathbb Z^{n+m}$ be a vector with positive integer coordinates and $g\geq0,n_1,n_2>0$ fixed integers. Then the map $G_{(\textbf{d}^r;\textbf{d}^l),g}^{n_1,n_2}(\textbf{x},\textbf{y},\textbf{c})$ is piecewise quasipolynomial in each chamber induced by $\tilde{\mathcal{H}}^{n_1,n_2}$ and each non-zero polynomial piece has degree $g$.
\end{theorem}

As a direct consequence we have the following

\begin{corollary}\label{cor-quasipolynomiality}
    Let $(\textbf{d}^r;\textbf{d}^l)\in\mathbb Z^{n+m}$ be a vector with positive integer coordinates and $g\geq0,n_1,n_2>0$ fixed integers and $\textbf c=(\textbf c^r;\textbf c^l)\in\mathbb Z^{n+m}$ such that $c_1^r>\dots>c_n^r$ and $c_1^l<\dots<c_m^l$. The map $G_{(\textbf d^r;\textbf d^l),\textbf c,g}^{n_1,n_2}(\textbf x,\textbf y)$ of double tropical Welschinger invariants is quasipolynomial in each chamber of $\Lambda\backslash\mathcal{H}^{n_1,n_2}(\textbf{c})$.
\end{corollary}

In the last section, we introduce new numbers defined via $s$-real multiplicity of real floor diagrams (see \cref{sec:real_diagrams} for a detailed discussion). This multiplicity depends on the imaginary part of the floor diagram and allows us to extend the enumerative framework beyond the totally real case. The new numbers coincide with the double tropical Welschinger invariants when the imaginary part is empty. As before, we encode these numbers into a map $G_{(\textbf{d}^r;\textbf{d}^l),\textbf{c},g,s}^{m_1,n_2,m_3}(\textbf{x},\textbf{y},\textbf{z},\textbf{w})$ and prove an analog of \cref{thm-quasipolynomiality}. Here, the variables $\textbf{z},\textbf{w}$ take into account the information from the imaginary part.

\subsection{Organization}\label{subsec:organization}

The paper is structured as follows: in \cref{sec:preliminaries} we introduce the definition of real floor diagrams and provide a proof of \cref{thm-correspondence}. Moreover, we give an overview on some results in weighted Ehrhart theory and weighted vector partition functions. In \cref{sec:pqp}, we prove \cref{thm-quasipolynomiality} and compute explicitly the maps $G_{(\textbf{d}^r;\textbf{d}^l),\textbf{c},g}^{n_1,n_2}(\textbf{x},\textbf{y})$ and $G_{(\textbf{d}^r;\textbf{d}^l),g}^{n_1,n_2}(\textbf{x},\textbf{y},\textbf{c})$ in an example. In \cref{sec:combinatorial_game}, we define new numbers using the multiplicity of floor diagrams. Then, we state and prove the main theorem on the piecewise quasipolynomiality of these numbers. Finally, we provide an example to illustrate the result.

\subsection{Acknowledgements}
The author would like to thank T. Blomme for helpful clarifications during the study, and M. A. Hahn for many valuable discussions and for his guidance during the writing of this work. The author thanks also an anonymous referee for many useful suggestions.

\section{Preliminaries}\label{sec:preliminaries}

\subsection{Real floor diagrams}\label{sec:real_diagrams}

In this section, we provide the definition of \textit{real floor diagram}. The idea is to adapt the definition given in \cite[Section 4]{arroyo2011recursive} to our problem. We use the same notation of \cite{ardila2013universal,ardila2017double,hahn2024universal}. Let $\mathsf{R}$ and $\mathsf{L}$ be the two multisets that contain the right and left directions:
\begin{equation}
    \mathsf{R}=\{\underbrace{c_1^r,\dots,c_1^r}_{d_1^r-times},\dots,\underbrace{c_n^r,\dots,c_n^r}_{d_n^r-times}\}\qquad \mathsf{L}=\{\underbrace{c_1^l,\dots,c_1^l}_{d_1^l-times},\dots,\underbrace{c_m^l,\dots,c_m^l}_{d_m^l-times}\}
\end{equation}
Note that $|\mathsf{R}|=|\mathsf{L}|=a$. Let $r=(r_1,\dots,r_a)$ and $l=(l_1,\dots,l_a)$ be permutations of the multisets $\mathsf{R}$ and $\mathsf{L}$ respectively.
\begin{definition}[\textbf{Floor diagram}]\label{def-marked_diagram}
    Let $\textbf{c},\textbf{d},a$ as above and $n_1,n_2,g$ nonnegative integers. A \textbf{floor diagram} $\mathcal D$ \textbf{of multidegree d}, type $(n_1,n_2)$ and genus $g$ for $P(\textbf{c},\textbf{d})$ is a connected weighted oriented and colored graph without loops and oriented cycles consisting of:
    \begin{enumerate}
        \item a vertex set $V=L\cup C\cup R$ such that $L\cup R$ contains $n_1$ white vertices and $C$ contains $a$ black vertices, $n_2$ white vertices and it is totally ordered from left to right. We denote the set of black vertices by $BV(\mathcal{D})$;
        \item a set $E$ of edges directed from left to right and such that every white vertex is incident to exactly one edge, which connects it to a black vertex;
        \item a map $w:E\to \mathbb Z_{>0}$ such that if we define the divergence of $v\in V$ to be
        \begin{equation}
            \text{div}(v)=\sum_{e:v\to v'}w(e)-\sum_{e:v'\to v}w(e),
        \end{equation}
        then $\text{div}(b_i)=r_i-l_i$, where $b_i\in C$ is the $i$-th black vertex in $C$ for all $i=1,\dots,a$ and
        \begin{align}
            &\sum_{v\in L\cup (C\setminus BV(\mathcal D))^{-}}\mathrm{div}(v)=d^t,\qquad\sum_{v\in R\cup (C\setminus BV(\mathcal D))^{+}}\mathrm{div}(v)=-\biggl(d^t+\sum_{i=1}^nc_i^rd_i^r-\sum_{j=1}^mc_j^ld_j^l\biggr),
        \end{align}
        where $(C\setminus BV(\mathcal D))^{-}$ (respectively $(C\setminus BV(\mathcal D))^{+}$) denotes the set of white vertices in $C$ with outgoing edge directed to the right (to the left);
        \item the first Betti number of $\mathcal D$ is $g$.
    \end{enumerate}
\end{definition}

Throughout the paper we refer to the vector $r-l$ as \textbf{divergence sequence of black vertices}.\\
We can attach to a floor diagram $\mathcal{D}$ of multidegree \textbf d, type $(n_1,n_2)$ and genus $g$ for $P(\textbf{c},\textbf{d})$ a vector of eight sequences $(\alpha,\beta,\gamma,\delta,\tilde\alpha,\tilde\beta,\tilde\gamma,\tilde\delta)$ with only finitely many non-zero terms that we call \textbf{divergence multiplicity vector}. The sequences $\alpha,\beta,\tilde\alpha,\tilde\beta$ were introduced in \cref{subsec:double_invariants}, while the sequences $\gamma,\delta,\tilde\gamma,\tilde\delta$ are considered here in order to deal with the imaginary part of the floor diagram, which will be defined later. Furthermore, they must satisfy the following equations:
\begin{align}
    \sum_ii[\tilde\alpha_i+\tilde\beta_i+2(\tilde\gamma_i+\tilde\delta_i)]=d^t,\qquad\sum_ii[\alpha_i+\beta_i+2(\gamma_i+\delta_i)]=d^t+\sum_{i=1}^nc_i^rd_i^r-\sum_{j=1}^mc_j^ld_j^l.
\end{align}
In what follows we write $|\alpha|$ for $\sum_{i\geq1}\alpha_i$. We introduce now a marking for a floor diagram $\mathcal D$ of multidegree \textbf d, type $(n_1,n_2)$ and genus $g$ for $P(\textbf{c},\textbf{d})$.

\begin{definition}[\textbf{Marking}]
    Let $n=n_1+n_2+2a+g-1$ be the number of white vertices, black vertices and black-black edges. A map $m:[n]\to\mathcal D$ is called \textbf{marking} for $\mathcal{D}$ if it satisfies the following conditions:
    \begin{itemize}
        \item $m$ is injective and respects the ordering of the vertices in $C$;
        \item if $m(i)>m(j)$ then $i>j$;
        \item if
        \begin{equation}
            \sum_{j=1}^{k-1}\tilde\alpha_j+1\leq i\leq\sum_{j=1}^k\tilde\alpha_j\qquad\text{or}\qquad|\tilde\alpha|+2\sum_{j=1}^{k-1}\tilde\gamma_j+1\leq i\leq|\tilde\alpha|+2\sum_{j=1}^k\tilde\gamma_j
        \end{equation}
        then $m(i)$ belongs to $L$ and has divergence $k$;
        \item for any $k\geq1$, there are exactly $\tilde\beta_k+2\tilde\delta_k$ white-black edges in $C$ that are in the image of $m$;
        \item let $\kappa=n-|\alpha+2\gamma|$, if
        \begin{equation}
            \kappa+\sum_{j=1}^{k-1}\alpha_j+1\leq i\leq \kappa+\sum_{j=1}^k\alpha_j\quad\text{or}\quad \kappa+|\alpha|+2\sum_{j=1}^{k-1}\gamma_j+1\leq i\leq \kappa+|\alpha|+2\sum_{j=1}^k\gamma_j
        \end{equation}
        then $m(i)$ belongs to $R$ and has divergence $-k$;
        \item for any $k\geq1$, there are exactly $\beta_k+2\delta_k$ black-white edges in $C$ that are in the image of $m$.
    \end{itemize}
\end{definition}

\begin{definition}[\textbf{Marked floor diagram}]
    Let $\mathcal D$ be a floor diagram of multidegree \textbf d, type $(n_1,n_2)$ and genus $g$ for $P(\textbf{c},\textbf{d})$ and $m$ be a marking for $\mathcal D$. We say that the floor diagram $\mathcal{D}$ is marked by $m$ and call it \textbf{marked floor diagram}. We denote it as $(\mathcal D,m)$.
\end{definition}

\begin{remark}
    It is important to stress the fact that the definition of marked floor diagram here differs from the one given in \cite{ardila2017double,hahn2024universal}: we do \textbf{not} consider gray vertices. The motivation behind this choice will be clear after reading \cref{def-real_multiplicity}. Strictly speaking, in the definition of \textit{real multiplicity} we do not want the weight of black-black edges to be squared. 
\end{remark}
The goal now is to define the imaginary part of a marked floor diagram. In the same fashion of \cite{arroyo2011recursive}, we introduce first $s$-pairs.
\begin{definition}[\textbf{$s$-pair}]\label{def-r_pair}
    Let $(\mathcal D,m)$ be a marked floor diagram and let us fix $s\geq0$ such that $n_2+2a+g-1-2s\geq0$. Note that $n_2+2a+g-1$ is the number of white vertices, black vertices and black-black edges in the central block. The set $\{i,i+1\}$ is called an $s$-pair if $i$ satisfies one of the following
    \begin{itemize}
        \item $i=|\tilde\alpha|+2k-1$ with $1\leq k\leq|\tilde\gamma|$;
        \item $i=|\tilde\alpha+2\tilde\gamma|+2k-1$ with $1\leq k\leq s$;
        \item $i=\kappa+|\alpha|+2k-1$ with $1\leq k\leq|\gamma|$.
    \end{itemize}
\end{definition}
\begin{example}
    Let us consider $n_1=4$, $n_2=2$, $g=0$, $a=2$ and $\beta=\gamma=\delta=\tilde\beta=0$, $\alpha=\tilde\alpha=1$, $\tilde\gamma=0001$ and $\tilde\delta=01$. Then $n_2+2a+g-1=5$ and $s$ can only be $0,1$ or $2$. If $s=0$ the only $0$-pair is given by $\{2,3\}$ since $|\tilde\gamma|=1$. If $s=1$, then $\{2,3\}$ and $\{4,5\}$ are both $1$-pairs. Finally, if $s=2$, then $\{2,3\}$, $\{4,5\}$ and $\{6,7\}$ are $2$-pairs.
\end{example}
\begin{remark}
    Notice that the number of $s$-pairs is $s+|\gamma+\tilde\gamma|$.
\end{remark}
If $\mathcal D$ is a floor diagram marked by $m$ and $s$ is as in \cref{def-r_pair}, the imaginary part of $\mathcal D$, denoted by $\mathcal I(\mathcal D,m,s)$, is given by
\begin{equation}
    \mathcal I(\mathcal D,m,s)=\{m(i)|\{i,i+1\}\text{ is an $s$-pair and }m(i)\text{ is not adjacent to }m(i+1)\}.
\end{equation}
We define a bijection $\rho_{\mathcal D,m}:\mathcal D\to\mathcal D$ as follows
\begin{itemize}
    \item $\rho_{\mathcal D,m}(m(i))=m(i)$ if $m(i)\in\mathcal D\setminus\mathcal I(\mathcal D,m,s)$;
    \item $\rho_{\mathcal D,m}(m(i))=m(j)$ if $\{i,j\}$ is an $s$-pair and $\{m(i),m(j)\}\subset\mathcal I(\mathcal D,m,s)$.
\end{itemize}
In particular, the function $\rho_{\mathcal D,m}$ is an involution. We associate to a marked floor diagram $\mathcal D$ a vector $(\textbf{x},\textbf{y},\textbf{z},\textbf{w})\in\mathbb{Z}^{m_1}\times\mathbb{Z}^{m_2}\times\mathbb Z^{m_3}\times\mathbb Z^{m_4}$ with $m_1+2m_3=n_1$ and $m_2+2m_4=n_2$, called \textbf{divergence sequence} where
\begin{itemize}
    \item $\textbf{x}=(\text{div}(\tilde q_1),\dots,\text{div}(\tilde q_{e}),\text{div}(q_1),\dots,\text{div}(q_{b}))$ is the sequence of divergences of white vertices in $L\cap(\mathcal D\setminus\mathcal I(\mathcal D,m,s))$ and $R\cap(\mathcal D\setminus\mathcal I(\mathcal D,m,s))$;
    \item $\textbf{y}$ is the sequence of divergences of white vertices in $C\cap(\mathcal D\setminus\mathcal I(\mathcal D,m,s))$;
    \item $\textbf{z}=(\text{div}(\tilde p_1),\dots,\text{div}(\tilde p_{\tilde e}),\text{div}(p_1),\dots,\text{div}(p_{\tilde b}))$ is the sequence of divergences of white vertices in $L\cap\mathcal I(\mathcal D,m,s)$ and $R\cap\mathcal I(\mathcal D,m,s)$;
    \item $\textbf{w}$ is the sequence of divergences of white vertices in $C\cap\mathcal I(\mathcal D,m,s)$.
\end{itemize}
We have a relation between the divergence sequence $(\textbf{x},\textbf{y},\textbf{z},\textbf{w})$ and the divergence multiplicity vector $(\alpha,\beta,\gamma,\delta,\tilde\alpha,\tilde\beta,\tilde\gamma,\tilde\delta)$ given by the following conditions:
\begin{multicols}{2}
    \begin{itemize}
        \item $\alpha_i$ is the number of entries $x_j=-i$.
        \item $\beta_i$ is the number of entries $y_j=-i$.
        \item $\gamma_i$ is the number of entries $z_j=-i$.
        \item $\delta_i$ is the number of entries $w_j=-i$.
        \item $\tilde\alpha_i$ is the number of entries $x_j=i$.
        \item $\tilde\beta_i$ is the number of entries $y_j=i$.
        \item $\tilde\gamma_i$ is the number of entries $z_j=i$.
        \item $\tilde\delta_i$ is the number of entries $w_j=i$.
    \end{itemize}
\end{multicols}
Since the sum of all the divergences in the graph must be $0$, we must have 
\begin{equation}
    \sum_{i=1}^{m_1}x_{i}+\sum_{j=1}^{m_2}y_{j}+\sum_{i=1}^{m_3}2z_{i}+\sum_{j=1}^{m_4}2w_{j}=d^t-d^b=\sum_{j=1}^mc_j^ld_j^l-\sum_{i=1}^nc_i^rd_i^r.
\end{equation}
\begin{remark}\label{rmk-delta_beta_difference}
    Before introducing the notion of $s$-real floor diagrams, we further clarify the distinction between the sequences $\beta,\tilde\beta$ and $\delta,\tilde\delta$. These pairs encode the distribution of positively and negatively directed white-black edges in the central block of the floor diagram, but with an important difference:
    \begin{itemize}
        \item The $i$-th entry of the sequences $\beta,\tilde\beta$ count real white-black edges of weight $i$ not involved in imaginary pairs.
        \item In contrast, the sequences $\delta,\tilde\delta$ track white-black edges that are identified through the involution $\rho_{\mathcal D,m}$ and thus contribute to the imaginary part of the diagram.
    \end{itemize}
\end{remark}
\begin{definition}[\textbf{$s$-real floor diagram}]\label{def-real_diagram}
    Let $(\mathcal D,m)$ be a marked floor diagram having divergence sequence $(\alpha,\beta,\gamma,\delta,\tilde\alpha,\tilde\beta,\tilde\gamma,\tilde\delta)$ and $s\geq0$. We say that $(\mathcal D,m)$ is a \textbf{marked} $s$-\textbf{real floor diagram} if $(\mathcal D,m)$ and $(\mathcal D,\rho_{\mathcal D,m}\circ m)$ are equivalent, namely there is a isomorphism of weighted oriented graphs between $(\mathcal D,m)$ and $(\mathcal D,\rho_{\mathcal D,m}\circ m)$, and there are exactly $2\delta_k$ black-white internal edges and $2\tilde\delta_k$ white-black internal edges of weight $k$ belonging to $\mathcal I(\mathcal D,m,s)$. 
\end{definition}
Strictly speaking, a marked floor diagram $(\mathcal D,m)$ is a marked $s$-real floor diagram if the bijection $\rho_{\mathcal D,m}$ exchanges pairs of consecutive edges having the same weight or pairs of consecutive black vertices having same divergence.
\begin{definition}[\textbf{$s$-real multiplicity}]\label{def-real_multiplicity}  
    Let $(\mathcal D,m)$ be a marked floor diagram having divergence sequence $(\alpha,\beta,\gamma,\delta,\tilde\alpha,\tilde\beta,\tilde\gamma,\tilde\delta)$. If $(\mathcal D,m)$ is an $s$-real floor diagram with all edges of even weight in $\mathcal I(\mathcal D,m,s)$, we define the $s$-\textbf{real multiplicity} of $(\mathcal D,m)$, denoted by $\mu_s(\mathcal D,m)$, as
    \begin{equation}
        \mu_s(\mathcal D,m)=(-1)^{\frac{|BV(\mathcal D)\cap\mathcal I(\mathcal D,m,s)|}{2}}\prod w(e)
    \end{equation}
     where the product runs over all the internal edges $e\in E$ such that $e\notin m(\{|\tilde\alpha+2\tilde\gamma|+2s+1,\dots,n\})$, and $\mu_s(\mathcal D,m)=0$ otherwise.
\end{definition}
\begin{remark}
    The definition of $s$-real floor diagrams given in this section is an adaptation of the one given in \cite[Section 4]{arroyo2011recursive}. Note that, if $s=0$, the function $\rho_{\mathcal D,m}$ exchanges the order of pairs of elements in $L$ and $R$, therefore any floor diagram is also a $0$-real floor diagram. In particular, the $0$-real multiplicity takes only two possible values: $0$ or $1$. To prove \cref{thm-quasipolynomiality}, we are interested in the \textit{totally real case}, i.e. the case in which the imaginary part of floor diagrams is empty. In terms of divergence multiplicity vector, this condition translates in asking for the sequences $\gamma,\tilde\gamma,\delta,\tilde\delta$ to be $0$.\\
    More precisely, the central block of a $0$-real floor diagram has empty intersection with the imaginary part, which forces the sequences $\delta,\tilde\delta$ to be $0$. Furthermore, since we are interested in the case in which the imaginary part is empty, also the sequences $\gamma,\tilde\gamma$ must be $0$.
\end{remark}
\begin{remark}
    From now on we refer to \textit{marked real floor diagrams} simply as \textit{floor diagrams}.
\end{remark}
\begin{example}\label{example-s-real-multiplicity}
    Let us consider the following data: $n_1=4$, $n_2=2$, $g=0$, $d_1^r=d_2^r=1$, $d^l=a=2$ and 
    \begin{equation}
        \mathsf{R}=\{c_1,c_2\}\qquad \mathsf{L}=\{0,0\}
    \end{equation}
    where $c_1,c_2\in\mathbb Z$ such that $c_1> c_2$. Consider the permutations $r=(c_1,c_2)$ and $l=(0,0)$, so we have $r-l=(c_1,c_2)$. The floor diagram $\mathcal D$ in \cref{fig:example_floor_diagram} has multidegree $\textbf{d}=(13;1,1;2)$, genus $0$ and it is of type $(4,2)$. The marking $m:[9]\to\mathcal D$ is given by red numbers in \cref{fig:example_floor_diagram}.
    \begin{figure}
        \centering

        \tikzset{every picture/.style={line width=0.75pt}} %set default line width to 0.75pt        
    
        \tikzset{every picture/.style={line width=0.75pt}} %set default line width to 0.75pt        
    
        \begin{tikzpicture}[x=0.75pt,y=0.75pt,yscale=-1,xscale=1]
            %uncomment if require: \path (0,300); %set diagram left start at 0, and has height of 300
    
            %Straight Lines [id:da45573938188752083] 
            \draw  [dash pattern={on 4.5pt off 4.5pt}]  (104,61) -- (105,197) ;
            %Straight Lines [id:da5615639411265632] 
            \draw  [dash pattern={on 4.5pt off 4.5pt}]  (422,62) -- (424,181) ;
            %Shape: Circle [id:dp5572904475436538] 
            \draw   (60,83.5) .. controls (60,79.36) and (63.36,76) .. (67.5,76) .. controls (71.64,76) and (75,79.36) .. (75,83.5) .. controls (75,87.64) and (71.64,91) .. (67.5,91) .. controls (63.36,91) and (60,87.64) .. (60,83.5) -- cycle ;
            %Shape: Circle [id:dp12649423975155338] 
            \draw   (60,187.5) .. controls (60,183.36) and (63.36,180) .. (67.5,180) .. controls (71.64,180) and (75,183.36) .. (75,187.5) .. controls (75,191.64) and (71.64,195) .. (67.5,195) .. controls (63.36,195) and (60,191.64) .. (60,187.5) -- cycle ;
            %Shape: Circle [id:dp6588283066289355] 
            \draw   (61,143.5) .. controls (61,139.36) and (64.36,136) .. (68.5,136) .. controls (72.64,136) and (76,139.36) .. (76,143.5) .. controls (76,147.64) and (72.64,151) .. (68.5,151) .. controls (64.36,151) and (61,147.64) .. (61,143.5) -- cycle ;
            %Shape: Circle [id:dp3505681746999192] 
            \draw   (187,121.5) .. controls (187,117.36) and (190.36,114) .. (194.5,114) .. controls (198.64,114) and (202,117.36) .. (202,121.5) .. controls (202,125.64) and (198.64,129) .. (194.5,129) .. controls (190.36,129) and (187,125.64) .. (187,121.5) -- cycle ;
            %Shape: Circle [id:dp5318254524159085] 
            \draw   (457,120.5) .. controls (457,116.36) and (460.36,113) .. (464.5,113) .. controls (468.64,113) and (472,116.36) .. (472,120.5) .. controls (472,124.64) and (468.64,128) .. (464.5,128) .. controls (460.36,128) and (457,124.64) .. (457,120.5) -- cycle ;
            %Shape: Circle [id:dp5391550380646064] 
            \draw   (128,121.5) .. controls (128,117.36) and (131.36,114) .. (135.5,114) .. controls (139.64,114) and (143,117.36) .. (143,121.5) .. controls (143,125.64) and (139.64,129) .. (135.5,129) .. controls (131.36,129) and (128,125.64) .. (128,121.5) -- cycle ;
            %Shape: Circle [id:dp6517832362972478] 
            \draw  [fill={rgb, 255:red, 0; green, 0; blue, 0 }  ,fill opacity=1 ] (261,122.5) .. controls (261,118.36) and (264.36,115) .. (268.5,115) .. controls (272.64,115) and (276,118.36) .. (276,122.5) .. controls (276,126.64) and (272.64,130) .. (268.5,130) .. controls (264.36,130) and (261,126.64) .. (261,122.5) -- cycle ;
            %Shape: Circle [id:dp7262353510252366] 
            \draw  [fill={rgb, 255:red, 0; green, 0; blue, 0 }  ,fill opacity=1 ] (360,122) .. controls (360,117.86) and (363.36,114.5) .. (367.5,114.5) .. controls (371.64,114.5) and (375,117.86) .. (375,122) .. controls (375,126.14) and (371.64,129.5) .. (367.5,129.5) .. controls (363.36,129.5) and (360,126.14) .. (360,122) -- cycle ;
            %Curve Lines [id:da43209117867087743] 
            \draw    (75,83.5) .. controls (200,57) and (366,69) .. (360,122) ;
            %Curve Lines [id:da372250565706405] 
            \draw    (76,143.5) .. controls (134,164) and (234,168) .. (261,122.5) ;
            %Curve Lines [id:da7420387181444426] 
            \draw    (75,187.5) .. controls (135,202) and (370,206) .. (360,122) ;
            %Straight Lines [id:da3798951872786487] 
            \draw    (276,122.5) -- (360,122) ;
            %Straight Lines [id:da05599550117650676] 
            \draw    (202,121.5) -- (261,122.5) ;
            %Curve Lines [id:da4332641171093461] 
            \draw    (143,121.5) .. controls (156,98) and (228,82) .. (261,122.5) ;
            %Straight Lines [id:da9469444723082019] 
            \draw    (375,122) -- (457,120.5) ;

            % Text Node
            \draw (192,80.4) node [anchor=north west][inner sep=0.75pt]    {$2$};
            % Text Node
            \draw (218,121.4) node [anchor=north west][inner sep=0.75pt]    {$2$};
            % Text Node
            \draw (292,99.4) node [anchor=north west][inner sep=0.75pt]    {$-4-c_{2}$};
            % Text Node
            \draw (160,161.4) node [anchor=north west][inner sep=0.75pt]    {$4$};
            % Text Node
            \draw (210,201.4) node [anchor=north west][inner sep=0.75pt]    {$4$};
            % Text Node
            \draw (64,79) node [anchor=north west][inner sep=0.75pt]    {$\textcolor[rgb]{0.82,0.01,0.01}{\text{x}}$};
            % Text Node
            \draw (64,139) node [anchor=north west][inner sep=0.75pt]    {$\textcolor[rgb]{0.82,0.01,0.01}{\text{x}}$};
            % Text Node
            \draw (63,183) node [anchor=north west][inner sep=0.75pt]    {$\textcolor[rgb]{0.82,0.01,0.01}{\text{x}}$};
            % Text Node
            \draw (163,99) node [anchor=north west][inner sep=0.75pt]    {$\textcolor[rgb]{0.82,0.01,0.01}{\text{x}}$};
            % Text Node
            \draw (228,117) node [anchor=north west][inner sep=0.75pt]    {$\textcolor[rgb]{0.82,0.01,0.01}{\text{x}}$};
            % Text Node
            \draw (264,117) node [anchor=north west][inner sep=0.75pt]    {$\textcolor[rgb]{0.82,0.01,0.01}{\text{x}}$};
            % Text Node
            \draw (313,117) node [anchor=north west][inner sep=0.75pt]    {$\textcolor[rgb]{0.82,0.01,0.01}{\text{x}}$};
            % Text Node
            \draw (363,117) node [anchor=north west][inner sep=0.75pt]    {$\textcolor[rgb]{0.82,0.01,0.01}{\text{x}}$};
            % Text Node
            \draw (460,117) node [anchor=north west][inner sep=0.75pt]    {$\textcolor[rgb]{0.82,0.01,0.01}{\text{x}}$};
            % Text Node
            \draw (62,53.4) node [anchor=north west][inner sep=0.75pt]    {$\textcolor[rgb]{0.82,0.01,0.01}{1}$};
            % Text Node
            \draw (64,118.4) node [anchor=north west][inner sep=0.75pt]    {$\textcolor[rgb]{0.82,0.01,0.01}{2}$};
            % Text Node
            \draw (60,201.4) node [anchor=north west][inner sep=0.75pt]    {$\textcolor[rgb]{0.82,0.01,0.01}{3}$};
            % Text Node
            \draw (154,84.4) node [anchor=north west][inner sep=0.75pt]    {$\textcolor[rgb]{0.82,0.01,0.01}{4}$};
            % Text Node
            \draw (218,104.4) node [anchor=north west][inner sep=0.75pt][font=\footnotesize]{$\textcolor[rgb]{0.82,0.01,0.01}{5}$};
            % Text Node
            \draw (263,133.4) node [anchor=north west][inner sep=0.75pt]    {$\textcolor[rgb]{0.82,0.01,0.01}{6}$};
            % Text Node
            \draw (314,129.4) node [anchor=north west][inner sep=0.75pt]    {$\textcolor[rgb]{0.82,0.01,0.01}{7}$};
            % Text Node
            \draw (369.5,132.9) node [anchor=north west][inner sep=0.75pt]    {$\textcolor[rgb]{0.82,0.01,0.01}{8}$};
            % Text Node
            \draw (466.5,131.4) node [anchor=north west][inner sep=0.75pt]    {$\textcolor[rgb]{0.82,0.01,0.01}{9}$};

        \end{tikzpicture}
        
        \caption{}
        \label{fig:example_floor_diagram}
    \end{figure}

    We can attach two divergence multiplicity vectors to $\mathcal D$:
    \begin{align}
        &\xi=(\alpha,\beta,\gamma,\delta,\tilde\alpha,\tilde\beta,\tilde\gamma,\tilde\delta)=(1,0,0,0,1,0,0001,01)\\
        &\xi'=(\alpha',\beta',\gamma',\delta',\tilde\alpha',\tilde\beta',\tilde\gamma',\tilde\delta')=(1,0,0,0,1,02,0001,0).
    \end{align}
    The divergence sequence associated to $\xi$ is $(\textbf{x}_\xi;\textbf{y}_\xi;\textbf{z}_\xi;\textbf{w}_\xi)=(1,-1;0;4;2)$, while the divergence sequence associated to $\xi'$ is $(\textbf{x}_{\xi'};\textbf{y}_{\xi'};\textbf{z}_{\xi'};\textbf{w}_{\xi'})=(1,-1;2,2;4;0)$.\\
    Let us compute the multiplicity of $\mathcal D$ with divergence multiplicity vector $\xi$. Note that, since there are some edges in $C$ having even weight we have that $\mu_0(\mathcal D)=0$. Let us consider the cases $s=1,2$:
    \begin{itemize}
        \item $s=1$: we have that $\{2,3\}$ and $\{4,5\}$ are $s$-pairs and since $4$ and $5$ are not adjacent, we have that
        \begin{equation}
            \mathcal I(\mathcal D,m,1)=\{2,3,4,5\}.
        \end{equation}
        Consider the bijection $\rho_{\mathcal D,m}:\mathcal D\to\mathcal D$ such that $\rho_{\mathcal D,m}(2)=3$ and $\rho_{\mathcal D,m}(4)=5$. Then $(\mathcal D,\rho_{\mathcal D,m}\circ m)$ and $(\mathcal D,m)$ are equivalent and we have $\mu_1(\mathcal D,m)=4$ as long as $c_2$ is odd and $-4-c_2>0$.
        \item $s=2$: we have that $\{2,3\}$, $\{4,5\}$ and $\{6,7\}$ are $s$-pairs, but $6$ is adjacent to $7$ and therefore $\mathcal I(\mathcal D,m,2)$ does not change from the previous case as well as $\rho_{\mathcal D,m}$, hence we get $\mu_2(\mathcal D,m)=4(-4-c_2)$ as long as $c_2$ is odd and $-4-c_2>0$.
    \end{itemize}
    It is easy to see that the $s$-real multiplicity of $\mathcal D$ with divergence multiplicity vector $\xi'$ is zero for all $s=0,1,2$.
\end{example}

The following theorem proves that we can enumerate floor diagrams instead of tropical curves.
\begin{theorem}\label{thm-correspondence}
    Let $\textbf d=(d^t;\textbf d^r;\textbf d^l)$ be a vector of positive integer numbers, $g\geq0$ an integer and \textbf{x} a vector with coordinates in $\mathbb{Z}\setminus\{0\}$. We write $\alpha(\textbf x)=\alpha$ and $\tilde\alpha(\textbf x)=\tilde\alpha$. Then, for any two sequences of non-negative integer numbers $\beta=(\beta_i)_{i\geq1}$ and $\tilde\beta=(\tilde\beta_i)_{i\geq1}$ such that
    \begin{equation}
        \sum_ii(\alpha_i+\beta_i)=d^t+\sum_{i=1}^nc_i^rd_i^r-\sum_{j=1}^mc_j^ld_j^l\quad\text{and}\quad\sum_ii(\tilde\alpha_i+\tilde\beta_i)=d^t,
    \end{equation}
    one has
    \begin{equation}
        W_{\textbf c,g}^{\alpha,\beta,\tilde\alpha,\tilde\beta}(\textbf d)=\sum_{\mathcal D}\mu(\mathcal D)
    \end{equation}
    where the sum runs over all floor diagrams $\mathcal D$ of multidegree \textbf d, genus $g$, left-right sequence \textbf x, and divergence multiplicity vector $(\alpha,\beta,\tilde\alpha,\tilde\beta)$ for $P=P(\textbf{c},\textbf{d})$.
\end{theorem}
\begin{proof}
    In \cite{hahn2024universal}, we illustrate the correspondence between tropical curves and floor diagrams. Therefore, what remains to prove here is just that the tropical curve $T$ and the corresponding floor diagram have the same Welschinger multiplicity: this follows from \cite[Remark 6]{shustin2012tropical} since in our situation the floor decomposition tells us that $|\text{int}(\Delta_v)\cap\mathbb Z^2|=0$ for each vertex $v$ of $T$, where $\Delta_v$ represents the triangle associated to $v$ in the subdivision of $P$ given by $T$. Hence, the multiplicity of $T$ can assume only two values, i.e. $0$ or $1$, and the claim follows from the correspondence between tropical curves and floor diagrams. In particular, when $\tilde\alpha,\tilde\beta=0$ we get the relative tropical Welschinger invariants defined in \cite{itenberg2009caporaso}. 
\end{proof}
\begin{remark}
    In the case of the projective plane and $\alpha=\tilde\alpha=\beta=\tilde\beta=0$ \cref{thm-correspondence} is the correspondence theorem stated and proved in \cite[Theorem 3]{mikhalkin2005enumerative}. Correspondence theorems are fundamental in the application of tropical geometry to enumerative problems, see \cite{mikhalkin2005enumerative} and \cite{shustin2004tropical,shustin2012tropical,ardila2017double,hahn2024universal,mandel2020descendant} for modification.
\end{remark}

\subsection{Weighted Ehrhart theory}\label{subsec:Ehrhart}

The focus of this section is the study of a particular weighted partition function. Specifically, we study the properties of a weighted partition function defined by means of a quasipolynomial. To this aim, we refer to results proved in \cite{de2024sums,sturmfels1995vector}, which are stated herein for completeness.
\begin{definition}\label{def-lattice_vectors}
    Let $X=\{a_1,\dots,a_m\}\subset\mathbb Z^d$ be a multiset of lattice vectors in $\mathbb R^d$. Sometimes we intend $X$ as a $d\times m$ matrix with coefficients in $\mathbb Z$.
    \begin{itemize}
        \item The \textbf{rank} of $X$, denoted by $\text{rank}(X)$, is defined as
        \begin{equation}
            \text{rank}(X)=\text{dim}(\text{Span}_\mathbb R(X)).
        \end{equation}
        \item $X$ is \textbf{unimodular} if all the maximal minors of $X$ are $-1,0$ or $1$.
        \item $X$ is \textbf{pointed} if $\text{cone}(X)$ does not contain a nontrivial linear subspace of $\mathbb R^d$.
        \item The \textbf{chamber complex} of $X$, denoted by $\text{Ch}(X)$, is the following set
        \begin{equation}
            \text{Ch}(X)=\{\sigma_Y=\text{cone}(Y)\subseteq\mathbb R^d|Y\subseteq X\}.
        \end{equation}
    \end{itemize}
\end{definition}
\begin{definition}\label{def-quasipolynomial}
    The function $f\colon\mathbb Z^d\to\mathbb R$ is \textbf{quasipolynomial} if there exists a full rank sublattice $\Lambda\subseteq\mathbb Z^d$ and $N$ different cosets $\Lambda_1,\dots,\Lambda_N$ of $\Lambda$ in $\mathbb Z^d$, where $N$ is the index of $\Lambda$ in $\mathbb Z^d$, such that $f(v)=f_i(v)$ for $v\in\Lambda_i$, where the $f_i$'s are polynomials. Furthermore, if $X\subset\mathbb Z^d$ is a pointed multiset of lattice vectors, $f$ is \textbf{piecewise quasipolynomial relative to} $\text{Ch}(X)$ if the restriction of $f$ to any $\sigma\subset\text{Ch}(X)$ is quasipolynomial.
\end{definition}
Let $c\in\mathbb Z^d$ and $X=\{a_1,\dots,a_m\}\subset\mathbb Z^d$ be a pointed vector configuration. Define the polytope
\begin{equation}\label{eq-polytope_PX(c)}
    P_X(c)\colon=\{z=(z_1,\dots,z_m)\in\mathbb R^m|Xz=c,\>z_i\geq0\text{ for all }i\in[m]\}
\end{equation}
and consider the function
\begin{equation}
    \mathcal P_X(c)=|P_X(c)\cap\mathbb Z^m|.
\end{equation}
This function is called \textbf{vector partition function}.
\begin{theorem}\label{thm-pqp_vector_partition_map}
    Let $X\subset\mathbb Z^d$ be a pointed vector configuration. The function $\mathcal P_X(c)$ is piecewise quasipolynomial relative to the chambers of $\text{Ch}(X)$. Furthermore, if $X$ is unimodular, then $\mathcal P_X(c)$ is piecewise polynomial relative to the chambers of $\text{Ch}(X)$ and each polynomial piece has degree $|X|-\text{rank}(X)$.
\end{theorem}
A proof of \cref{thm-pqp_vector_partition_map} can be found in \cite{sturmfels1995vector}.
\begin{definition}\label{def-parametric}
    Let $z=(z_1,\dots,z_m)\in\mathbb Z^m$. A polytope $Q$ is called \textbf{parametric polytope} if
    \begin{equation}
        Q=Q(z_1,\dots,z_m)=\biggl\{y\in\mathbb R^k\bigg|Cy=\sum_{i=1}^mz_id_i+e,\>y_j\geq0\text{ for all }j\in[k]\biggr\}
    \end{equation}
    with $C\in\mathbb Z^{r\times k}$, $d_i,e\in\mathbb Z^r$.
\end{definition}
\begin{definition}\label{def-Ehrhartqp}
    A function $f:\mathbb Z^m\to\mathbb Z$ is called \textbf{Ehrhart quasipolynomial} if $f$ is quasipolynomial and there exists a parametric polytope $Q(z_1,\dots,z_m)\subset\mathbb R^k$ such that
    \begin{equation}
        f(z_1,\dots,z_m)=|Q(z_1,\dots,z_m)\cap\mathbb Z^k|.
    \end{equation}
    Let $f:\mathbb Z^m\to\mathbb Z$ be an Ehrhart quasipolynomial and $X=\{a_1,\dots,a_m\}\subset\mathbb Z^d$ a pointed vector configuration. Define the \textbf{weighted partition function} as
    \begin{equation}
        \mathcal P_{X,f}(c)=\sum_{z\in P_X(c)\cap\mathbb Z^m}f(z).
    \end{equation}
\end{definition}
The following theorem provides a method for converting a weighted sum of lattice points of a polytope into the enumeration of lattice points of a higher-dimensional polytope. \cref{thm-weight_lifting_polytope} is proved in full generality in \cite{de2024sums}.
\begin{theorem}\label{thm-weight_lifting_polytope}
    Let $X=\{a_1,\dots,a_m\}\subset\mathbb Z^d$ be a pointed vector configuration, $c\in\mathbb Z^d$ and consider the polytope $P_X(c)$. Let $w(z_1,\dots,z_m)$ be any function such that
    \begin{equation}
        w(z_1,\dots,z_m)=|Q(z_1,\dots,z_m)\cap\mathbb Z^k|.
    \end{equation} 
    Then there exists a \textbf{weight lifting polytope}
    \begin{equation}
        P_X^\ast(c)=\biggl\{(z,y)\in\mathbb R^{m+k}\bigg|
        \begin{pmatrix}
            & &X & &0\\
            &d_1 &\dots &d_m &-C
        \end{pmatrix}
        \begin{pmatrix}
            z\\
            y
        \end{pmatrix}
        =
        \begin{pmatrix}
            c\\
            -e
        \end{pmatrix},\>
        z_i,y_j\geq0 \text{ for all }i\in[m],j\in[k]\biggr\}
    \end{equation}
    such that
    \begin{equation}
        \mathcal P_{X,w}(c)=\sum_{z\in P_X(c)\cap\mathbb Z^m}w(z)=|P_X^*(c)\cap\mathbb Z^{m+k}|,
    \end{equation}
    where $C,d_1,\dots,d_m,e$ are as in \cref{def-parametric}.
\end{theorem}
In what follows we derive two key technical propositions. Consider the lattice $\Lambda$ given by all the vectors in $\mathbb Z^m$ having even entries and let denote by $\mathbbm 1\in\mathbb Z^m$ the vector with entries equal to $1$. Let 
\begin{equation}
    \Lambda_{[m]}=\mathbbm 1+\Lambda,\qquad\Lambda_I=e_I+\Lambda
\end{equation}
be the cosets of $\Lambda$ in $\mathbb Z^m$. Here, $I=\{i_1,\dots,i_k\}\subset[m]$ of size $k=0,1,\dots,m-1$ and $e_I=\sum_{i\in I}e_i$ where $e_i$ is the vector having $1$ in the $i$-th entry and $0$ elsewhere. Define $\pi_X:\mathbb Z^m\to\mathbb R$ as the function
\begin{equation}\label{eq-map}
    \pi_X(z)=
    \begin{cases}
        1 &\text{if }z\in\Lambda_{[m]}\text{ and } z_i\geq0\\
        0 &\text{if }z\in\Lambda_I\text{ for some }I\subset [m].
    \end{cases}
\end{equation}\vspace*{-\baselineskip}
\begin{proposition}\label{cor-Ehrhart_quasipoly}
    The weighted vector partition function $\mathcal P_{X,\pi_X}(c)$ is piecewise quasipolynomial relative to the chamber complex $\text{Ch}(X)$. In particular, each nonzero quasipolynomial piece has degree $|X|-\text{rank}(X)$.
\end{proposition}
\begin{proof}
    First of all, note that $\pi_X(z_1,\dots,z_m)=|Q(z_1,\dots,z_m)\cap\mathbb Z^m|$, where
        \begin{equation}
        Q(z_1,\dots,z_m)=\{y\in\mathbb Z^m|2y_i=z_i-1,\>y_i\geq0\text{ for all }i\in[m]\}.
    \end{equation}
    By \cref{thm-weight_lifting_polytope}, the function $\mathcal P_{X,\pi_X}(c)$ coincides with the vector partition function $\mathcal P_{\tilde X}(\tilde c)$, where $\tilde X=\{(a_1,e_1),\dots,(a_m,e_m),(0,-2e_1),\dots,(0,-2e_m)\}\subset\mathbb Z^{d+m}$ and $\tilde c=(c,\sum e_i)\in\mathbb Z^{d+m}$. Note that $\tilde X$ is a pointed vector configuration. Indeed, if $V\subseteq\mathbb R^{m+d}$ is a vector space such that $V\subseteq\text{cone}(\tilde X)$ and $\{v_1,\cdots,v_n\}$ is a basis of $V$, then
    \begin{equation}
        v_i=\sum_{j=1}^m\lambda_{ij}(a_j,e_j)+\sum_{j=1}^m\mu_{ij}(0,-2e_j)=\sum_{j=1}^m(\lambda_{ij}a_j,(\lambda_{ij}-2\mu_{ij})e_j)\qquad\text{for all }i\in[n]
    \end{equation}
    where $\lambda_{ij},\mu_{ij}\in\mathbb R_{\geq0}$ for all $j\in[m]$. Since $V$ is a vector space, then $-v_i\in V\subseteq\text{cone}(\tilde X)$, but $X=\{a_1,\dots,a_m\}$ is a pointed vector configuration, hence $\lambda_{ij}=0$ for all $i\in[n]$ and $j\in[m]$, meaning that $V\subseteq\text{cone}(\tilde Y)$ where $\tilde Y=\{(0,-2e_1),\dots,(0,-2e_m)\}$ that is a contradiction. Thus the function $\mathcal P_{\tilde X}$ satisfies the hypothesis of \cref{thm-pqp_vector_partition_map}, so the function $\mathcal P_{\tilde X}$ is piecewise quasipolynomial relative to $\text{Ch}(\tilde{X})$. Consider $\pi:\mathbb R^{d+m}\to\mathbb R^d$ the projection on the first $d$ coordinates, then $\text{Ch}(X)=\pi(\text{Ch}(\tilde X))$ and in particular $\mathcal P_{\tilde X}$ restricted to $\mathbb R^d$ is piecewise quasipolynomial relative to $\text{Ch}(X)$.
\end{proof}
Let $k=2h$ be an even integer with $h\geq0$. Let $Y=\{a_{j_1},\dots,a_{j_k}\}\subseteq E=\{a_{t_1},\dots,a_{t_r}\}\subseteq X$ and $\Lambda$ as above. Let $J\subseteq\{j_1,\dots,j_k\}$ and define the following set
\begin{equation}
    \mathsf P_J=\{z\in(\mathbb Z_{>0})^m|z\in\Lambda_{[m]\setminus J}\text{ and }z_{j_i}=z_{j_{i+1}}\text{ for all }i=1,3,\dots,k-1\}
\end{equation}Define the function $\pi_Y:\mathbb Z^m\to\mathbb R$ as 
\begin{equation}
    \pi_Y(z)=
    \begin{cases}
        \displaystyle\prod_{i=1}^r z_{t_i} &\text{if }z\in\mathsf P_{J}\text{ for some } J\subseteq\{j_1,\dots,j_k\}\\
        0 &\text{if }z\in\Lambda_I\text{ for some } I \text{ with } I\subset[m]\setminus\{j_1,\dots,j_k\}.
    \end{cases}
\end{equation}\vspace*{-\baselineskip}
\begin{proposition}\label{cor-EhrhartY_qp}
    The weighted vector partition function $\mathcal P_{X,\pi_Y}(c)$ is piecewise quasipolynomial relative to the chamber complex $\text{Ch}(X)$.
\end{proposition}
\begin{proof}
    In order to avoid confusion, we consider $Y=\{a_{r-k+1},\dots,a_r\},\> E=\{a_1,\dots,a_r\}$. Let us consider the parametric polytope $Q(z_1,\dots,z_m)$ given by the vectors $y=(y_1,\dots,y_{m+2r-k})\in(\mathbb Z_{\geq0})^{m+2r-k}$ such that:
    \begin{itemize}
        \item $y_i+y_{m+i}=z_i-1$ for $1\leq i\leq r$;
        \item $2y_i=z_i-1$ for $r+1\leq i\leq m$;
        \item $2y_{m+r+i}=z_{i}-1$ for $1\leq i\leq r-k$;
        \item $y_{r-k+i}+y_{m+r-k+i}-(y_{r-k+i+1}+y_{m+r-k+i+1})=0$ for $i=1,3,\dots,k-1$.
    \end{itemize}
%\begin{equation}
%    Q(z_1,\dots,z_m)=\biggl\{y\in\mathbb Z^{m+2r-k}\bigg|
%    \begin{cases}
%        2y_i=z_i+1 &\text{if }i\in[m]\text{ s. t. }a_i\in X\setminus Y\\
%        y_{i1}+y_{i2}=z_i+1 &\text{if }i\in[m]\text{ s. t. }a_i\in E
%    \end{cases}
%    \biggr\}\biggr\}.
%\end{equation}
    Observe that the first and last conditions together imply $z_{r-k+i}=z_{r-k+i+1}$. Note that $\pi_Y(z_1,\dots,z_m)=|Q(z_1,\dots,z_m)\cap\mathbb Z^{m+2r-k}|$. Hence, by \cref{thm-weight_lifting_polytope}, the function $\mathcal P_{X,\pi_Y}(c)$ equals the vector partition function $\mathcal P_{\tilde X}(\tilde c)$, where $\tilde X\subset\mathbb Z^{d+m+r-h}$ is the vector configuration decomposed as $\tilde X=\tilde X_1\cup\tilde X_2\cup\tilde X_3$ such that
    \small
    \begin{align}
        &\tilde X_1=\{(a_i,e_i+e_{m+i})\}_{i=1}^{r-k}\cup\{(a_i,e_i)\}_{i=r-k+1}^k\\
        \tilde X_2=\{(0,-e_i)\}_{i=1}^{r-k}\cup\biggl(\bigcup_{\substack{i=1\\i\text{ odd}}}^{k-1}&\{(0,-e_{r-k+i}+e_{m+r-k+i}),(0,-e_{r-k+i+1}-e_{m+r-k+i})\}\biggr)\cup\{(0,-2e_i)\}_{i=r+1}^m\\
        \tilde X_3=\{(0,-e_i)\}_{i=1}^{r-k}\cup\biggl(\bigcup_{\substack{i=1\\i\text{ odd}}}^{k-1}&\{(0,-e_{r-k+i}+e_{m+r-k+i}),(0,-e_{r-k+i+1}-e_{m+r-k+i})\}\biggr)\cup\{(0,-2e_{m+i})\}_{i=1}^{r-k}
    \end{align}
    \normalsize
    with $\tilde c=(c,\sum_{i=1}^{m+r-k} e_i,0)\in\mathbb Z^{d+m+r-h}$. The order of the blocks is not random and it is such that the column vectors in $\tilde X_1$ correspond to $z_1,\dots,z_m$, the column vectors in $\tilde X_2$ correspond to $y_1,\dots,y_m$ and the column vectors in $\tilde X_3$ correspond to $y_{m+1},\dots,y_{m+2r-k}$. Furthermore, the vector configuration $\tilde X$ is pointed. To see this, we only have to verify that there is no vector space contained in the cone generated by the vectors of the form
    \begin{equation}
        \bigcup_{\substack{i=1\\i\text{ odd}}}^{k-1}\{(0,-e_{r-k+i}+e_{m+r-k+i}),(0,-e_{r-k+i+1}-e_{m+r-k+i})\}.
    \end{equation}
    However, this immediately follows from the sign of $e_{r-k+i}$, $i=1,\dots,k$. The proof is then the same of \cref{cor-Ehrhart_quasipoly}.
\end{proof}

\section{Piecewise quasipolynomiality of double tropical Welschinger invariants}\label{sec:pqp}

\subsection{Notation}

Fix $g,n_1,n_2\in\mathbb Z_{\geq0}$, and $(\textbf{d}^r;\textbf{d}^l)\in(\mathbb Z_{>0})^{n+m}$. Let $(\mathcal D,m)$ be a floor diagram and denote by $\tilde{\mathcal D}$ the graph obtained from $\mathcal D$ by removing all the weights, but such that the underlying graph $\tilde{\mathcal D}$ inherits the partition $V=L\cup C\cup R$ of the vertices and the ordering given by $m$. We call $\mathcal G$ the collection of such graphs that contribute to $G_{(\textbf{d}^r;\textbf{d}^l),g}^{n_1,n_2}(\textbf{x},\textbf{y},\textbf{c})$. In particular, $\mathcal G$ is finite and depends only on $g,a$ and $n_1+n_2$. We denote by $Perm(\mathsf{R})$ and $Perm(\mathsf{L})$ the sets of permutations of the multisets $\mathsf{R}$ and $\mathsf{L}$ respectively and let $r(\textbf{c})\in Perm(\mathsf{R})$ and $l(\textbf{c})\in Perm(\mathsf{L})$. For each graph $(H,m)\in\mathcal G$, let $E(H)$ and $V(H)$ be the sets of edges and vertices of $H$ respectively and define the set $W_{H,r(\textbf{c})-l(\textbf{c})}(\textbf{x},\textbf{y},\textbf{c})$ of weights $w:E(H)\to\mathbb N$ for which the resulting weighted graph is a floor diagram for $P(\textbf{c},\textbf{d})$, i.e. such that the $i$-th black vertex has divergence $r_i(\textbf{c})-l_i(\textbf{c)}$ and white divergence sequence $(\textbf{x},\textbf{y})$. By construction, the obtained floor diagram has genus $g$ and multidegree \textbf{d}. Finally, call $\mathbb R^{X}=\{\textbf w:X\to\mathbb R\}$ and let $\pi_{E(H)}:\mathbb R^{E(H)}\to\mathbb R$ be the polynomial map defined by 
\begin{equation}
    \pi_{E(H)}(\textbf{w})=
    \begin{cases}
        1 &\text{if }\textbf{w}(e)\equiv1\bmod 2 \text{ for all }e\in E(H)\\
        0 &\text{otherwise.}
    \end{cases}
\end{equation}

\subsection{Proof of Theorem \ref{thm-quasipolynomiality}}\label{subsec:proof}

First, we note that the domain of the function $G_{(\textbf{d}^r;\textbf{d}^l),g}^{n_1,n_2}(\textbf{x},\textbf{y},\textbf{c})$ is 
\begin{equation}
    \Lambda\cap\{(\textbf{x},\textbf{y},\textbf{c})\in\mathbb Z^{n_1+n_2+n+m}|c_1^r>\dots>c_n^r,\>c_1^l<\dots<c_m^l\},
\end{equation}
where we see $\Lambda$ as a subset of $\mathbb Z^{n_1+n_2+n+m}$. Define
\begin{equation}
    G_{H,r(\textbf{c})-l(\textbf{c})}(\textbf{x},\textbf{y},\textbf{c})=\sum_{\textbf{w}\in W_{H,r(\textbf{c})-l(\textbf{c})}}\pi_{E(H)}(\textbf{w}).
\end{equation}
Note that $G_{H,r(\textbf{c})-l(\textbf{c})}(\textbf{x},\textbf{y},\textbf{c})$ depends on the order of the entries of \textbf{y}, while in $G_{(\textbf{d}^r;\textbf{d}^l),g}^{n_1,n_2}(\textbf{x},\textbf{y},\textbf{c})$ we have to consider all the distinct orders for $\textbf{y}$:
\begin{equation}
    G_{(\textbf{d}^r;\textbf{d}^l),g}^{n_1,n_2}(\textbf{x},\textbf{y},\textbf{c})=\frac{1}{\beta_1!\beta_2!\cdots\tilde\beta_1!\tilde\beta_2!\cdots}\sum_{(H,m)\in\mathcal G}\sum_{(r(\textbf{c}),l(\textbf{c}))}\sum_{\sigma\in S_{n_2}}G_{H,r(\textbf{c})-l(\textbf{c})}(\textbf{x},\sigma(\textbf{y}),\textbf{c})
\end{equation}
\textbf{Step 1: express $G_{(\textbf{d}^r;\textbf{d}^l),g}^{n_1,n_2}(\textbf x,\textbf y,\textbf{c})$ as a weighted partition function.} Recall that the divergence of a vertex is defined as 
\begin{equation}
    \text{div}(v)=\sum_{e:v\to v'}w(e)-\sum_{e:v'\to v}w(e)
\end{equation} 
and that the adjacency matrix of the graph $H$ is $A_H\in\mathbb R^{V(H)\times E(H)}$ which is, in our convention:
\begin{equation}
    A_H(v,e)=
    \begin{cases}
        1 &\text{when}\>e:v\to v'\>\text{for some}\>v'\\
        -1 &\text{when}\>e:v'\to v\>\text{for some}\>v'\\
        0 &\text{otherwise.}
    \end{cases}
\end{equation}
Note that the columns of the matrix $A_H$ are a subset of the root system $A_{|E(H)|-1}$ (see \cite[Example 4.4]{ardila2017double}). Now, take $\textbf{k}=\textbf{k}(\textbf{c})\in\mathbb R^{V(H)}$ and define the flow polytope
\begin{align}
    \Phi_H(\textbf{k}(\textbf{c}))&=\{\textbf{w}\in\mathbb R^{E(H)}|\>\textbf{w}(e)\geq0\>\text{for all}\> e\in E(H),\>\text{div}(v)=\textbf{k}(v)\>\text{for all vertices}\>v\}\\
    &=\{\textbf{w}\in\mathbb R^{E(H)}|\>A_H\textbf{w}=\textbf{k}(\textbf{c}),\>\textbf{w}\geq0\}.
\end{align}
If $\textbf{k}(\textbf{c})$ is the vector which entries are given by $(\textbf{x},\textbf{y})$ for the white vertices and $r(\textbf{c})-l(\textbf{c})$ for the black vertices, then $W_{H,r(\textbf{c})-l(\textbf{c})}(\textbf{x},\textbf{y},\textbf{c})=\Phi_H(\textbf{k}(\textbf{c}))\cap\mathbb Z^{E(H)}$.\\
Define the weighted partition function 
\begin{equation}
    \mathcal P_{H,\pi_{E(H)}}(\textbf{k}(\textbf{c}))=\sum_{\textbf{w}\in\Phi_H(\textbf{k}(\textbf{c}))\cap\mathbb Z^{E(H)}}\pi_{E(H)}(\textbf{w})
\end{equation}
and consider the hyperplane $\{\textbf{k}\in\mathbb R^{V(H)}|\sum\textbf{k}(v)=0\}$. Let $H_{r(\textbf{c})-l(\textbf{c})}\subset\{\textbf{k}\in\mathbb R^{V(H)}|\sum\textbf{k}(v)=0\}$ be the subspace determined by the equations
\begin{equation}
    \textbf{k}(w_i)=x_i,\quad \textbf{k}(w_j)=y_j,\quad \textbf{k}(b_i)=r_i(\textbf{c})-l_i(\textbf{c})\>\text{for all black}\> b_i.
\end{equation}
The restriction of $\mathcal P_{H,\pi_{E(H)}}(\textbf{k}(\textbf{c}))$ to the subspace $H_{r(\textbf{c})-l(\textbf{c})}$ is the map $G_{H,r(\textbf{c})-l(\textbf{c})}(\textbf{x},\textbf{y},\textbf{c})$.\\
\textbf{Step 2: the map $G_{(\textbf{d}^r;\textbf{d}^l),g}^{n_1,n_2}(\textbf{x},\textbf{y},\textbf{c})$ is piecewise quasipolynomial.} In the notation of \cref{cor-Ehrhart_quasipoly}, $X$ is given by the columns of $A_H$, which do not depend on the vector $\textbf{c}$. Therefore, by \cref{cor-Ehrhart_quasipoly}, the weighted partition function $\mathcal P_{H,\pi_{E(H)}}(\textbf{k}(\textbf{c}))$ is piecewise quasipolynomial relative to $\mathrm{Ch}(X)$. By \cite[Example 4.4]{ardila2017double}, the chambers in $\mathrm{Ch}(X)$ coincide with the chambers of the discriminant arrangement in $\{\textbf{k}\in\mathbb R^{V(H)}|\sum\textbf{k}(v)=0\}$. Recall that this arrangement consists of the hyperplanes
\begin{equation}
    \sum_{v'\in V'}\textbf{k}(v')=0\qquad\text{for all subsets }V'\subseteq V.
\end{equation}
In particular, $G_{H,r(\textbf{c})-l(\textbf{c})}(\textbf{x},\textbf{y},\textbf{c})$ is piecewise quasipolynomial relative to the chambers of the discriminant arrangement in $H_{r(\textbf{c})-l(\textbf{c})}$. Denote this discriminant arrangement by $S_{r(\textbf{c})-l(\textbf{c})}$. When we symmetrise, the result $\displaystyle\sum_{\sigma\in S_{n_2}}G_{H,\textbf{c},r-l}(\textbf{x},\sigma(\textbf{y}))$ is still piecewise quasipolynomial relative to the same chambers, since the chamber structure is fixed under permutation of the $n_2$ \textbf{y} variables. What remains to prove is that 
\begin{equation}
    \sum_{(r(\textbf{c}),l(\textbf{c}))}\sum_{\sigma\in S_{n_2}}G_{H,r(\textbf{c})-l(\textbf{c})}(\textbf{x},\sigma(\textbf{y}),\textbf{c})
\end{equation}
is piecewise quasipolynomial. For fixed \textbf{c}, when we sum over all the pairs $(r(\textbf{c}),l(\textbf{c}))$, the resulting map is piecewise quasipolynomial relative to the chambers of the common refinement of the hyperplane arrangements $S_{r(\textbf{c})-l(\textbf{c})}$, in other words $\mathcal{H}^{n_1,n_2}(\textbf{c})=\displaystyle\bigcup_{(r(\textbf{c}),l(\textbf{c}))}S_{r(\textbf{c})-l(\textbf{c})}$. Thus, when we let \textbf{c} vary, the resulting function is piecewise quasipolynomial relative to the chambers induced by $\tilde{\mathcal{H}}^{n_1,n_2}$.\\
\textbf{Step 3: $G_{(\textbf{d}^r;\textbf{d}^l),g}^{n_1,n_2}(\textbf{x},\textbf{y},\textbf{c})$ has degree $g$.} Note that
\begin{equation}
    \text{dim}(\Phi_H(\textbf{k}))=|E(H)|-\text{rank}(A_H)=|E(H)|-(|E(H)|-g)=g
\end{equation}
and since each non-zero polynomial piece of $\pi_{E(H)}$ has degree $0$, we have that $\mathcal{P}_{H,\pi_{E(H)}}$ has degree $g$, hence $G_{(\textbf{d}^r;\textbf{d}^l),g}^{n_1,n_2}(\textbf{x},\textbf{y},\textbf{c})$ has degree $g$.

\subsection{Example}\label{subsec:example1}

In this section, we provide an example that illustrates the result of \cref{thm-quasipolynomiality}. Let us consider the following data: let $k\in\mathbb Z_{\geq0}$ and consider $\textbf{c}^r=(k,0)$, $\textbf{c}^l=(0,0)$, $\textbf{d}^r=(1,1)$, $\textbf{d}^l=d^l=2$, $a=2$, $n_1=1$, $n_2=2$ and $g=1$. The $h$-transverse polygon we are considering for this example is depicted in \cref{fig:h-transverse_polygon}. We have the following multisets
\begin{equation}
    \mathsf{R}=\{k,0\}\qquad \mathsf{L}=\{0,0\}
\end{equation}
that originate the permutations $r_1=(k,0)$, $r_2=(0,k)$ and $l=(0,0)$. Therefore, black vertices can have divergences given by
\begin{equation}
    r_1-l=(k,0),\qquad r_2-l=(0,k).
\end{equation}
We compute the maps $G^{1,2}_{(\textbf{d}^r;\textbf{d}^l),\textbf{c},1}(x_1,y_1,y_2)$ and $G^{1,2}_{(\textbf{d}^r;\textbf{d}^l),1}(x_1,y_1,y_2,k)$ in only one chamber of the hyperplane arrangement. The hyperplane arrangement in
\begin{equation}
    \Lambda=\{(x_1,y_1,y_2)\in\mathbb Z^3|x_1+y_1+y_2+k=0\}
\end{equation}
is given by hyperplanes
\begin{align}
    x_1=0,\qquad &x_1+k=0,\qquad y_1=0,\qquad y_1+k=0,\qquad y_2=0,\qquad y_2+k=0.
\end{align}

\begin{figure}
    \centering
    \tikzset{every picture/.style={line width=0.75pt}} %set default line width to 0.75pt        

        \begin{tikzpicture}[x=0.75pt,y=0.75pt,yscale=-1,xscale=1]
        %uncomment if require: \path (0,300); %set diagram left start at 0, and has height of 300

        %Shape: Polygon [id:ds17488778839130092] 
        \draw   (211,120) -- (210,180) -- (62,180) -- (62,91) -- (152,91) -- cycle ;

        % Text Node
        \draw (259,98) node [anchor=north west][inner sep=0.75pt]
            {$ \begin{array}{l}
                \mathbf{c}^{r} =( k,0) ,\ \mathbf{c}^{l} =( 0,0)\\
                \mathbf{d}^{r} =( 1,1) ,\mathbf{d}^{l} =2\\
                a=2,\ n_{1} =1,\ n_{2} =2
            \end{array}$};

    \end{tikzpicture}

    \caption{}
    \label{fig:h-transverse_polygon}
\end{figure}

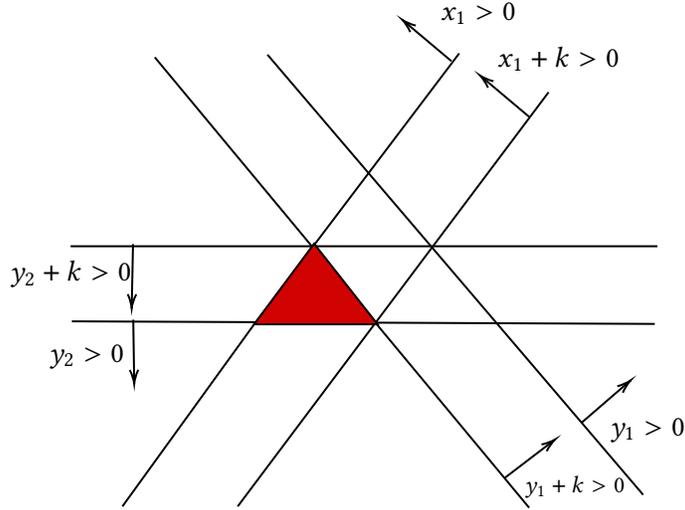
\begin{figure}
    \centering
    \tikzset{every picture/.style={line width=0.75pt}} %set default line width to 0.75pt        

    \begin{tikzpicture}[x=0.70pt,y=0.70pt,yscale=-0.70,xscale=0.70]
        %uncomment if require: \path (0,708); %set diagram left start at 0, and has height of 708
    
        %Straight Lines [id:da1392665430229072] 
        \draw    (390,452) -- (491.71,311.3) -- (650,102) ;
        %Straight Lines [id:da6693041887501123] 
        \draw    (479,452) -- (719,132) ;
        %Straight Lines [id:da647155129306513] 
        \draw    (350,251) -- (803,251) ;
        %Straight Lines [id:da9579159021602939] 
        \draw    (351,309) -- (801,311) ;
        %Straight Lines [id:da14595530485539365] 
        \draw    (792,443) -- (502,103) ;
        %Straight Lines [id:da4946947233704948] 
        \draw    (704,453) -- (415,105) ;
        %Straight Lines [id:da2994470874358284] 
        \draw    (644,108) -- (607.53,77.29) ;
        \draw [shift={(606,76)}, rotate = 40.1] [color={rgb, 255:red, 0; green, 0; blue, 0 }  ][line width=0.75]    (10.93,-3.29) .. controls (6.95,-1.4) and (3.31,-0.3) .. (0,0) .. controls (3.31,0.3) and (6.95,1.4) .. (10.93,3.29)   ;
        %Straight Lines [id:da587505475862875] 
        \draw    (704,151) -- (667.53,120.29) ;
        \draw [shift={(666,119)}, rotate = 40.1] [color={rgb, 255:red, 0; green, 0; blue, 0 }  ][line width=0.75]    (10.93,-3.29) .. controls (6.95,-1.4) and (3.31,-0.3) .. (0,0) .. controls (3.31,0.3) and (6.95,1.4) .. (10.93,3.29)   ;
        %Straight Lines [id:da7396797328116171] 
        \draw    (744.83,387) -- (780.91,355.83) ;
        \draw [shift={(782.42,354.53)}, rotate = 139.18] [color={rgb, 255:red, 0; green, 0; blue, 0 }  ][line width=0.75]    (10.93,-3.29) .. controls (6.95,-1.4) and (3.31,-0.3) .. (0,0) .. controls (3.31,0.3) and (6.95,1.4) .. (10.93,3.29)   ;
        %Straight Lines [id:da5193072970569319] 
        \draw    (684.76,430.28) -- (722.64,401.33) ;
        \draw [shift={(724.23,400.12)}, rotate = 142.62] [color={rgb, 255:red, 0; green, 0; blue, 0 }  ][line width=0.75]    (10.93,-3.29) .. controls (6.95,-1.4) and (3.31,-0.3) .. (0,0) .. controls (3.31,0.3) and (6.95,1.4) .. (10.93,3.29)   ;
        %Straight Lines [id:da3790203465520465] 
        \draw    (398.2,249.35) -- (397.22,297.02) ;
        \draw [shift={(397.18,299.02)}, rotate = 271.18] [color={rgb, 255:red, 0; green, 0; blue, 0 }  ][line width=0.75]    (10.93,-3.29) .. controls (6.95,-1.4) and (3.31,-0.3) .. (0,0) .. controls (3.31,0.3) and (6.95,1.4) .. (10.93,3.29)   ;
        %Straight Lines [id:da5492089170150714] 
        \draw    (398.39,307.74) -- (399.09,355.42) ;
        \draw [shift={(399.12,357.42)}, rotate = 269.16] [color={rgb, 255:red, 0; green, 0; blue, 0 }  ][line width=0.75]    (10.93,-3.29) .. controls (6.95,-1.4) and (3.31,-0.3) .. (0,0) .. controls (3.31,0.3) and (6.95,1.4) .. (10.93,3.29)   ;
        %Shape: Triangle [id:dp8024624317878942] 
        \draw  [fill={rgb, 255:red, 208; green, 2; blue, 2 }  ,fill opacity=1 ] (538,249) -- (587,311.3) -- (491.71,311.3) -- cycle ;

        % Text Node
        \draw (634,60) node [anchor=north west][inner sep=0.75pt]    {$x_{1}  >0$};
        % Text Node
        \draw (678,94) node [anchor=north west][inner sep=0.75pt]    {$x_{1} +k >0$};
        % Text Node
        \draw (763.26,381.1) node [anchor=north west][inner sep=0.75pt]    {$y_{1}  >0$};
        % Text Node
        \draw (697.69,425.86) node [anchor=north west][inner sep=0.75pt]    [font=\footnotesize]{$y_{1} +k >0$};
        % Text Node
        \draw (300.19,259.62) node [anchor=north west][inner sep=0.75pt]    {$y_{2} +k >0$};
        % Text Node
        \draw (330.4,324.23) node [anchor=north west][inner sep=0.75pt]    {$y_{2}  >0$};

    \end{tikzpicture}
    
    \caption{Hyperplane arrangement in $\Lambda$}
    \label{fig:hyperplane_arrangement}
\end{figure}

Here, we consider the chamber $\mathbf{C}$ given by the inequalities $-k<x_1<0$, $y_1<-k<0$ and $-k<y_2<0$, which is the red chamber in \cref{fig:hyperplane_arrangement}. First, we compute the map $G^{1,2}_{(\textbf{d}^r;\textbf{d}^l),\textbf{c},1}(x_1,y_1,y_2)$ in $\mathbf{C}$ for $k$ even and $k$ odd and then we put all together to obtain the map $G^{1,2}_{(\textbf{d}^r;\textbf{d}^l),1}(x_1,y_1,y_2,k)$.\\
Let us assume $k$ \textbf{even}. In order to have non-zero multiplicity, $x_1,y_1$ and $y_2$ must be odd. However, since $k$ is even and $k=-x_1-y_1-y_2$, at least one among $x_1,y_1$ and $y_2$ is even. Therefore, $G^{1,2}_{(\textbf{d}^r;\textbf{d}^l),\textbf{c},1}(x_1,y_1,y_2)=0$ for $k$ even.\\
Let us assume $k$ \textbf{odd}. If $x_1,y_1,y_2$ are odd, then $\pm (x_1+k)$, $\pm(y_1+k)$ and $\pm(y_2+k)$ are all even, hence $-w\pm(x_1+k)$, $-w\pm(y_1+k)$ and $-w\pm(y_2+k)$ are odd if and only if $w$ is odd. On the other hand $-w\pm x_1$, $-w\pm y_1$, $-w\pm y_2$ and $-w+k$ are odd if and only if $w$ is even, therefore the floor diagrams $A1$, $B2$, $B3$, $C1$, $C3$, $D1$ and $D2$ in \cref{table1} have multiplicity $0$.

For $x_1,y_1,y_2,k$ odd we have
\begin{align}
    G^{1,2}_{(\textbf{d}^r;\textbf{d}^l),\textbf{c},1}(x_1,y_1,y_2)&=10\sum_{\substack{w=0\\w\text{ odd}}}^{y_2+k}1+10\sum_{\substack{w=0\\w\text{ odd}}}^{y_1+k}1+2\sum_{\substack{w=0\\w\text{ odd}}}^{x_1+k}1\\
    &=10\frac{y_2+k+1}{2}+10\frac{y_1+k+1}{2}+2\frac{x_1+k+1}{2}\\
    &=x_1+5y_1+5y_2+11k+11.
\end{align}\vspace*{-\baselineskip}

\begin{figure}
    \centering
    \tikzset{every picture/.style={line width=0.75pt}} %set default line width to 0.75pt        


    
    \caption{}
    \label{fig:markings}
\end{figure}

\begin{table}
    \centering
        
    \begin{tabular}{|p{0.02\textwidth}|p{0.29\textwidth}|p{0.29\textwidth}|p{0.29\textwidth}|}
        \hline 
        &\begin{center}
            $\displaystyle 1$
        \end{center}
        &\begin{center}
            $\displaystyle 2$
        \end{center}
        &\begin{center}
            $\displaystyle 3$
        \end{center}
        \\
\hline
    \end{tabular}

    \caption{Floor diagrams with divergence sequence $r_1-l$ contributing to $G^{1,2}_{(\textbf{d}^r;\textbf{d}^l),\textbf{c},1}(x_1,y_1,y_2)$ in $\mathbf{C}$.}
    \label{table1}
\end{table}

\begin{table}
    \centering
        
    \begin{tabular}{|p{0.02\textwidth}|p{0.29\textwidth}|p{0.29\textwidth}|p{0.29\textwidth}|}
        \hline 
        &\begin{center}
            $\displaystyle 1$
        \end{center}
        &\begin{center}
            $\displaystyle 2$
        \end{center}
        &\begin{center}
            $\displaystyle 3$
        \end{center}
        \\
\hline
        \end{tabular}
        
    \caption{Floor diagrams with divergence sequence $r_2-l$ contributing to $G^{1,2}_{(\textbf{d}^r;\textbf{d}^l),\textbf{c},1}(x_1,y_1,y_2)$ in $\mathbf{C}$.}
    \label{table2}
        
\end{table}\vspace*{-\baselineskip}

\begin{remark}
    The first two sums come from the multiplicities of the graphs $A2,A3$ and $B1$, while the last sum comes from the multiplicities of the graph $C2$. The factor in front of each summation is given by the possible markings of the floor diagrams. For instance, the floor diagram $B1$ has six possible markings as it is showed in \cref{fig:markings} and the floor diagrams $A2$ and $A3$ have two possible markings each.
\end{remark}
We can then write the map $G^{1,2}_{(\textbf{d}^r;\textbf{d}^l),1}(x_1,y_1,y_2,k)$ in the chamber $\mathbf{C}$
\begin{equation}
    G^{1,2}_{(\textbf{d}^r;\textbf{d}^l),1}(x_1,y_1,y_2,k)=
    \begin{cases}
        x_1+5y_1+5y_2+11k+11 &x_1,y_1,y_2,k\text{ odd}\\
        0 &\text{otherwise.}
    \end{cases}
\end{equation}

\section{Combinatorial game}\label{sec:combinatorial_game}

The search for Welschinger-type invariants for the enumeration of real curves of positive genera and passing through a configuration of points, that allows also complex conjugate points, on algebraic surfaces is an active research area \cite{shustin2015higher,itenberg2017welschinger,itenberg2018relative}. However, at present no such invariants are known, either in the algebraic setting or in the tropical one. In this section, we introduce new numbers, defined as multiplicities of $s$-real floor diagrams, which - while they fail to be invariants - provide interesting combinatorics. The goal of this section is to study their piecewise quasipolynomiality.\\
We first define an equivalence relation on vectors of sequences. Those vectors represent the divergence sequences of $s$-real floor diagrams. The idea is that vectors in the same equivalence class are attached to $s$-real floor diagrams for an $h$-transverse polygon $P(\textbf{c},\textbf{d})$ that have the same multidegree, type and genus, but with different imaginary part. The numbers we study in this section are then defined to be the sum of multiplicities of certain $s$-real floor diagrams having divergence sequence in a fixed equivalence class. We clarify the need of this equivalence relation in the following
\begin{example}
    Let us consider the same data from \cref{example-s-real-multiplicity} except for $\xi=(1,0,0,0,1,0,0001,001)$ and $\xi'=(1,0,0,0,1,002,0001,0)$. Let $\mathcal D_1$ and $\mathcal D_2$ the floor diagrams in \cref{fig:equivalence_relation}.
    \begin{figure}[H]
        \tikzset{every picture/.style={line width=0.75pt}} %set default line width to 0.75pt        

        \begin{tikzpicture}[x=0.70pt,y=0.70pt,yscale=-0.70,xscale=0.70]
        %uncomment if require: \path (0,688); %set diagram left start at 0, and has height of 688

            %Straight Lines [id:da8270263958446589] 
            \draw  [dash pattern={on 4.5pt off 4.5pt}]  (82,308) -- (83,444) ;
            %Straight Lines [id:da8206917222179047] 
            \draw  [dash pattern={on 4.5pt off 4.5pt}]  (400,309) -- (402,428) ;
            %Shape: Circle [id:dp1306920204475729] 
            \draw   (38,330.5) .. controls (38,326.36) and (41.36,323) .. (45.5,323) .. controls (49.64,323) and (53,326.36) .. (53,330.5) .. controls (53,334.64) and (49.64,338) .. (45.5,338) .. controls (41.36,338) and (38,334.64) .. (38,330.5) -- cycle ;
            %Shape: Circle [id:dp3541599110605126] 
            \draw   (38,434.5) .. controls (38,430.36) and (41.36,427) .. (45.5,427) .. controls (49.64,427) and (53,430.36) .. (53,434.5) .. controls (53,438.64) and (49.64,442) .. (45.5,442) .. controls (41.36,442) and (38,438.64) .. (38,434.5) -- cycle ;
            %Shape: Circle [id:dp7619734671020353] 
            \draw   (39,390.5) .. controls (39,386.36) and (42.36,383) .. (46.5,383) .. controls (50.64,383) and (54,386.36) .. (54,390.5) .. controls (54,394.64) and (50.64,398) .. (46.5,398) .. controls (42.36,398) and (39,394.64) .. (39,390.5) -- cycle ;
            %Shape: Circle [id:dp6294564272371498] 
            \draw   (165,368.5) .. controls (165,364.36) and (168.36,361) .. (172.5,361) .. controls (176.64,361) and (180,364.36) .. (180,368.5) .. controls (180,372.64) and (176.64,376) .. (172.5,376) .. controls (168.36,376) and (165,372.64) .. (165,368.5) -- cycle ;
            %Shape: Circle [id:dp35891654800339445] 
            \draw   (435,367.5) .. controls (435,363.36) and (438.36,360) .. (442.5,360) .. controls (446.64,360) and (450,363.36) .. (450,367.5) .. controls (450,371.64) and (446.64,375) .. (442.5,375) .. controls (438.36,375) and (435,371.64) .. (435,367.5) -- cycle ;
            %Shape: Circle [id:dp42955167356243795] 
            \draw   (106,368.5) .. controls (106,364.36) and (109.36,361) .. (113.5,361) .. controls (117.64,361) and (121,364.36) .. (121,368.5) .. controls (121,372.64) and (117.64,376) .. (113.5,376) .. controls (109.36,376) and (106,372.64) .. (106,368.5) -- cycle ;
            %Shape: Circle [id:dp08246628542722167] 
            \draw  [fill={rgb, 255:red, 0; green, 0; blue, 0 }  ,fill opacity=1 ] (239,369.5) .. controls (239,365.36) and (242.36,362) .. (246.5,362) .. controls (250.64,362) and (254,365.36) .. (254,369.5) .. controls (254,373.64) and (250.64,377) .. (246.5,377) .. controls (242.36,377) and (239,373.64) .. (239,369.5) -- cycle ;
            %Shape: Circle [id:dp08588533713045754] 
            \draw  [fill={rgb, 255:red, 0; green, 0; blue, 0 }  ,fill opacity=1 ] (338,369) .. controls (338,364.86) and (341.36,361.5) .. (345.5,361.5) .. controls (349.64,361.5) and (353,364.86) .. (353,369) .. controls (353,373.14) and (349.64,376.5) .. (345.5,376.5) .. controls (341.36,376.5) and (338,373.14) .. (338,369) -- cycle ;
            %Curve Lines [id:da5010973814205286] 
            \draw    (53,330.5) .. controls (178,304) and (344,316) .. (338,369) ;
            %Curve Lines [id:da4497346759798895] 
            \draw    (54,390.5) .. controls (112,411) and (212,415) .. (239,369.5) ;
            %Curve Lines [id:da2204739801066392] 
            \draw    (53,434.5) .. controls (113,449) and (348,453) .. (338,369) ;
            %Straight Lines [id:da6964977947586423] 
            \draw    (254,369.5) -- (338,369) ;
            %Straight Lines [id:da2608209592148879] 
            \draw    (180,368.5) -- (239,369.5) ;
            %Curve Lines [id:da9647609025458779] 
            \draw    (121,368.5) .. controls (134,345) and (206,329) .. (239,369.5) ;
            %Straight Lines [id:da8673177260050793] 
            \draw    (353,369) -- (435,367.5) ;
            %Straight Lines [id:da7864502867472004] 
            \draw  [dash pattern={on 4.5pt off 4.5pt}]  (567,302) -- (568,438) ;
            %Straight Lines [id:da7929998579030518] 
            \draw  [dash pattern={on 4.5pt off 4.5pt}]  (885,303) -- (887,422) ;
            %Shape: Circle [id:dp36321322558305325] 
            \draw   (523,324.5) .. controls (523,320.36) and (526.36,317) .. (530.5,317) .. controls (534.64,317) and (538,320.36) .. (538,324.5) .. controls (538,328.64) and (534.64,332) .. (530.5,332) .. controls (526.36,332) and (523,328.64) .. (523,324.5) -- cycle ;
            %Shape: Circle [id:dp5647329126579644] 
            \draw   (523,428.5) .. controls (523,424.36) and (526.36,421) .. (530.5,421) .. controls (534.64,421) and (538,424.36) .. (538,428.5) .. controls (538,432.64) and (534.64,436) .. (530.5,436) .. controls (526.36,436) and (523,432.64) .. (523,428.5) -- cycle ;
            %Shape: Circle [id:dp7779051118076318] 
            \draw   (524,384.5) .. controls (524,380.36) and (527.36,377) .. (531.5,377) .. controls (535.64,377) and (539,380.36) .. (539,384.5) .. controls (539,388.64) and (535.64,392) .. (531.5,392) .. controls (527.36,392) and (524,388.64) .. (524,384.5) -- cycle ;
            %Shape: Circle [id:dp3172643215194454] 
            \draw  [fill={rgb, 255:red, 0; green, 0; blue, 0 }  ,fill opacity=1 ] (650,362.5) .. controls (650,358.36) and (653.36,355) .. (657.5,355) .. controls (661.64,355) and (665,358.36) .. (665,362.5) .. controls (665,366.64) and (661.64,370) .. (657.5,370) .. controls (653.36,370) and (650,366.64) .. (650,362.5) -- cycle ;
            %Shape: Circle [id:dp2578049711359872] 
            \draw   (920,361.5) .. controls (920,357.36) and (923.36,354) .. (927.5,354) .. controls (931.64,354) and (935,357.36) .. (935,361.5) .. controls (935,365.64) and (931.64,369) .. (927.5,369) .. controls (923.36,369) and (920,365.64) .. (920,361.5) -- cycle ;
            %Shape: Circle [id:dp3965076754215262] 
            \draw   (591,362.5) .. controls (591,358.36) and (594.36,355) .. (598.5,355) .. controls (602.64,355) and (606,358.36) .. (606,362.5) .. controls (606,366.64) and (602.64,370) .. (598.5,370) .. controls (594.36,370) and (591,366.64) .. (591,362.5) -- cycle ;
            %Shape: Circle [id:dp3342124754288385] 
            \draw  [fill={rgb, 255:red, 255; green, 255; blue, 255 }  ,fill opacity=1 ] (724,363.5) .. controls (724,359.36) and (727.36,356) .. (731.5,356) .. controls (735.64,356) and (739,359.36) .. (739,363.5) .. controls (739,367.64) and (735.64,371) .. (731.5,371) .. controls (727.36,371) and (724,367.64) .. (724,363.5) -- cycle ;
            %Shape: Circle [id:dp22448540174909248] 
            \draw  [fill={rgb, 255:red, 0; green, 0; blue, 0 }  ,fill opacity=1 ] (823,363) .. controls (823,358.86) and (826.36,355.5) .. (830.5,355.5) .. controls (834.64,355.5) and (838,358.86) .. (838,363) .. controls (838,367.14) and (834.64,370.5) .. (830.5,370.5) .. controls (826.36,370.5) and (823,367.14) .. (823,363) -- cycle ;
            %Curve Lines [id:da5744678973741734] 
            \draw    (538,324.5) .. controls (663,298) and (829,310) .. (823,363) ;
            %Curve Lines [id:da06495913619948757] 
            \draw    (539,384.5) .. controls (597,405) and (631,382) .. (650,362.5) ;
            %Curve Lines [id:da10554428059146226] 
            \draw    (538,428.5) .. controls (598,443) and (833,447) .. (823,363) ;
            %Straight Lines [id:da11627782433377643] 
            \draw    (739,363.5) -- (823,363) ;
            %Straight Lines [id:da6507638247438139] 
            \draw    (606,362.5) -- (650,362.5) ;
            %Curve Lines [id:da6543181076943854] 
            \draw    (665,362.5) .. controls (696,333) and (764,323) .. (823,363) ;
            %Straight Lines [id:da8315851492731273] 
            \draw    (838,363) -- (920,361.5) ;

            % Text Node
            \draw (168,324) node [anchor=north west][inner sep=0.75pt]    {$3$};
            % Text Node
            \draw (194,365) node [anchor=north west][inner sep=0.75pt]    {$3$};
            % Text Node
            \draw (255,343) node [anchor=north west][inner sep=0.75pt]    {$-5-c_{2}$};
            % Text Node
            \draw (136,405) node [anchor=north west][inner sep=0.75pt]    {$4$};
            % Text Node
            \draw (186,445) node [anchor=north west][inner sep=0.75pt]    {$4$};
            % Text Node
            \draw (40,324) node [anchor=north west][inner sep=0.75pt]    {$\textcolor[rgb]{0.82,0.01,0.01}{\text{x}}$};
            % Text Node
            \draw (40,382) node [anchor=north west][inner sep=0.75pt]    {$\textcolor[rgb]{0.82,0.01,0.01}{\text{x}}$};
            % Text Node
            \draw (39,428) node [anchor=north west][inner sep=0.75pt]    {$\textcolor[rgb]{0.82,0.01,0.01}{\text{x}}$};
            % Text Node
            \draw (139,344) node [anchor=north west][inner sep=0.75pt]    {$\textcolor[rgb]{0.82,0.01,0.01}{\text{x}}$};
            % Text Node
            \draw (204,362) node [anchor=north west][inner sep=0.75pt]    {$\textcolor[rgb]{0.82,0.01,0.01}{\text{x}}$};
            % Text Node
            \draw (240,362) node [anchor=north west][inner sep=0.75pt]    {$\textcolor[rgb]{0.82,0.01,0.01}{\text{x}}$};
            % Text Node
            \draw (289,362) node [anchor=north west][inner sep=0.75pt]    {$\textcolor[rgb]{0.82,0.01,0.01}{\text{x}}$};
            % Text Node
            \draw (339,362) node [anchor=north west][inner sep=0.75pt]    {$\textcolor[rgb]{0.82,0.01,0.01}{\text{x}}$};
            % Text Node
            \draw (436,362) node [anchor=north west][inner sep=0.75pt]    {$\textcolor[rgb]{0.82,0.01,0.01}{\text{x}}$};
            % Text Node
            \draw (38,297) node [anchor=north west][inner sep=0.75pt]    {$\textcolor[rgb]{0.82,0.01,0.01}{1}$};
            % Text Node
            \draw (40,362) node [anchor=north west][inner sep=0.75pt]    {$\textcolor[rgb]{0.82,0.01,0.01}{2}$};
            % Text Node
            \draw (36,445) node [anchor=north west][inner sep=0.75pt]    {$\textcolor[rgb]{0.82,0.01,0.01}{3}$};
            % Text Node
            \draw (130,328) node [anchor=north west][inner sep=0.75pt]    {$\textcolor[rgb]{0.82,0.01,0.01}{4}$};
            % Text Node
            \draw (194,348) node [anchor=north west][inner sep=0.75pt]  [font=\footnotesize]  {$\textcolor[rgb]{0.82,0.01,0.01}{5}$};
            % Text Node
            \draw (239,377) node [anchor=north west][inner sep=0.75pt]    {$\textcolor[rgb]{0.82,0.01,0.01}{6}$};
            % Text Node
            \draw (290,373) node [anchor=north west][inner sep=0.75pt]    {$\textcolor[rgb]{0.82,0.01,0.01}{7}$};
            % Text Node
            \draw (345.5,376.5) node [anchor=north west][inner sep=0.75pt]    {$\textcolor[rgb]{0.82,0.01,0.01}{8}$};
            % Text Node
            \draw (442.5,375) node [anchor=north west][inner sep=0.75pt]    {$\textcolor[rgb]{0.82,0.01,0.01}{9}$};
            % Text Node
            \draw (609,363) node [anchor=north west][inner sep=0.75pt]    {$3$};
            % Text Node
            \draw (752,363) node [anchor=north west][inner sep=0.75pt]    {$3$};
            % Text Node
            \draw (702,315) node [anchor=north west][inner sep=0.75pt]    {$-8-c_{2}$};
            % Text Node
            \draw (616,388) node [anchor=north west][inner sep=0.75pt]    {$4$};
            % Text Node
            \draw (671,439) node [anchor=north west][inner sep=0.75pt]    {$4$};
            % Text Node
            \draw (525,317) node [anchor=north west][inner sep=0.75pt]    {$\textcolor[rgb]{0.82,0.01,0.01}{\text{x}}$};
            % Text Node
            \draw (525,377) node [anchor=north west][inner sep=0.75pt]    {$\textcolor[rgb]{0.82,0.01,0.01}{\text{x}}$};
            % Text Node
            \draw (524,422) node [anchor=north west][inner sep=0.75pt]    {$\textcolor[rgb]{0.82,0.01,0.01}{\text{x}}$};
            % Text Node
            \draw (620,356) node [anchor=north west][inner sep=0.75pt]    {$\textcolor[rgb]{0.82,0.01,0.01}{\text{x}}$};
            % Text Node
            \draw (651,356) node [anchor=north west][inner sep=0.75pt]    {$\textcolor[rgb]{0.82,0.01,0.01}{\text{x}}$};
            % Text Node
            \draw (704,332) node [anchor=north west][inner sep=0.75pt]    {$\textcolor[rgb]{0.82,0.01,0.01}{\text{x}}$};
            % Text Node
            \draw (774,356) node [anchor=north west][inner sep=0.75pt]    {$\textcolor[rgb]{0.82,0.01,0.01}{\text{x}}$};
            % Text Node
            \draw (824,356) node [anchor=north west][inner sep=0.75pt]    {$\textcolor[rgb]{0.82,0.01,0.01}{\text{x}}$};
            % Text Node
            \draw (921,356) node [anchor=north west][inner sep=0.75pt]    {$\textcolor[rgb]{0.82,0.01,0.01}{\text{x}}$};
            % Text Node
            \draw (523,291) node [anchor=north west][inner sep=0.75pt]    {$\textcolor[rgb]{0.82,0.01,0.01}{1}$};
            % Text Node
            \draw (525,356) node [anchor=north west][inner sep=0.75pt]    {$\textcolor[rgb]{0.82,0.01,0.01}{2}$};
            % Text Node
            \draw (521,439) node [anchor=north west][inner sep=0.75pt]    {$\textcolor[rgb]{0.82,0.01,0.01}{3}$};
            % Text Node
            \draw (620,335) node [anchor=north west][inner sep=0.75pt]    {$\textcolor[rgb]{0.82,0.01,0.01}{4}$};
            % Text Node
            \draw (657.5,370) node [anchor=north west][inner sep=0.75pt]  [font=\footnotesize]  {$\textcolor[rgb]{0.82,0.01,0.01}{5}$};
            % Text Node
            \draw (702,344) node [anchor=north west][inner sep=0.75pt]    {$\textcolor[rgb]{0.82,0.01,0.01}{6}$};
            % Text Node
            \draw (775,367) node [anchor=north west][inner sep=0.75pt]    {$\textcolor[rgb]{0.82,0.01,0.01}{7}$};
            % Text Node
            \draw (830.5,370.5) node [anchor=north west][inner sep=0.75pt]    {$\textcolor[rgb]{0.82,0.01,0.01}{8}$};
            % Text Node
            \draw (927.5,369) node [anchor=north west][inner sep=0.75pt]    {$\textcolor[rgb]{0.82,0.01,0.01}{9}$};
            % Text Node
            \draw (128,476) node [anchor=north west][inner sep=0.75pt]    {$\mathcal{D}_{1} ,\ \xi $};
            % Text Node
            \draw (658,475) node [anchor=north west][inner sep=0.75pt]    {$\mathcal{D}_{2} ,\ \xi '$};

        \end{tikzpicture}

        \caption{}
        \label{fig:equivalence_relation}
    \end{figure}

    Both $\mathcal D_1$ and $\mathcal D_2$ are $1$-real floor diagrams having same type, genus and multidegree with non-zero multiplicity. Therefore, they contribute to the same number, although their imaginary parts differ. Indeed $\mathcal I(\mathcal D_1,m_1,1)=\{m_1(2),m_1(3),m_1(4),m_1(5)\}$, while $\mathcal I(\mathcal D_2,m_2,1)=\{m_2(2),m_2(3)\}$. Let $\tilde\varepsilon=(0,0,-1)$, one can check that $\tilde\beta'=\tilde\beta-2\tilde\varepsilon$ and $\tilde\delta'=\tilde\delta+\tilde\varepsilon$. This operation corresponds to removing an $s$-pair from the imaginary part.
\end{example}
\begin{remark}
    \cref{example-s-real-multiplicity} shows that two divergence sequences in the same equivalence class cannot be attached to the same $s$-real floor diagram.
\end{remark}
We encode the combinatorial Welschinger-type numbers in a function $G_{(\textbf{d}^r;\textbf{d}^l),\textbf{c},g,r}^{n_1,n_2}(\textbf{x},\textbf{y},\textbf{z},\textbf{w})$, that depends on four vectors of variables, defined over a lattice. The vector $\textbf{z}$ traces the white vertices in the imaginary part of the floor diagrams that belong to the left and right block, while the vector $\textbf{w}$ traces the white vertices in the imaginary part of the floor diagrams that belong to the central block. Then, the equivalence relation on the vectors of sequences induces an equivalence relation on the vectors $(\textbf{x},\textbf{y},\textbf{z},\textbf{w})$. Therefore, we have a well-defined map on the quotient lattice that gets rid of the vector $\textbf{w}$. Finally, we use results from \cref{subsec:Ehrhart} to prove the piecewise quasipolynomiality.

Let us denote by $\mathcal C$ the set of sequences $\alpha=(\alpha_1,\alpha_2,\dots)$ such that $\alpha_i\in\mathbb Z$ and $\alpha_i\neq0$ for only finitely many $i$'s. We write $|\alpha|=\sum\alpha_i$ and define the operation $\alpha+\beta=(\alpha_1+\beta_1,\alpha_2+\beta_2,\dots)$. Let $\mathcal C^+$ be the subset of $\mathcal C$ of sequences $\alpha$ such that $\alpha_i\geq0$ for all $i$. We define a relation on $(\mathcal C^+)^8$ in the following way: let $\xi=(\alpha,\beta,\gamma,\delta,\tilde\alpha,\tilde\beta,\tilde\gamma,\tilde\delta)$ and $\xi'=(\alpha',\beta',\gamma',\delta',\tilde\alpha',\tilde\beta',\tilde\gamma',\tilde\delta')$ in $(\mathcal C^+)^8$, we say $\xi\sim\xi'$ if and only if $\alpha'=\alpha$, $\gamma'=\gamma$, $\tilde\alpha'=\tilde\alpha$, $\tilde\gamma'=\tilde\gamma$ and
\begin{align}
    &\beta'=\beta-2\varepsilon\qquad\tilde\beta'=\tilde\beta-2\tilde\varepsilon\\
    &\delta'=\delta+\varepsilon\qquad\>\>\>\>\tilde\delta'=\tilde\delta+\tilde\varepsilon
\end{align}
for some $\varepsilon,\tilde\varepsilon\in\mathcal C$ such that
\begin{align}
    \begin{cases}
        \text{if }\beta_j=\delta_j=0\implies\varepsilon_j=0\\
        \varepsilon_j=i\leq0\implies0\leq-i\leq\delta_j\\
        \varepsilon_j=i\geq0\implies0\leq i\leq\frac{\beta_j}{2}
    \end{cases}
    \qquad
    \begin{cases}
        \text{if }\tilde\beta_j=\tilde\delta_j=0\implies\tilde\varepsilon_j=0\\
        \tilde\varepsilon_j=i\leq0\implies0\leq-i\leq\tilde\delta_j\\
        \tilde\varepsilon_j=i\geq0\implies0\leq i\leq\frac{\tilde\beta_j}{2}
    \end{cases}
\end{align}
\begin{proposition}
    The relation $\sim$ is an equivalence relation on $(\mathcal C^+)^8$.
\end{proposition}
In the notation of \cref{sec:real_diagrams}, let $s\geq0$ such that $n_2+2a+g-1-2s\geq0$. Let $\textbf{d}=(d^t;\textbf{d}^r;\textbf{d}^l)$ be a vector with positive integer entries and $\textbf{c}=(\textbf{c}^r;\textbf{c}^l)\in\mathbb Z^{n+m}$ such that $c_1^r>\dots>c_n^r$ and $c_1^l<\dots<c_m^l$. Let $\xi=(\alpha,\beta,\gamma,\delta,\tilde\alpha,\tilde\beta,\tilde\gamma,\tilde\delta)\in(\mathcal C^+)^8$ such that
\begin{align}
    \sum_i i[\alpha_i+\beta_i+2(\gamma_i+\delta_i)]=d^t+\sum_{i=1}^nc_i^rd_i^r-\sum_{j=1}^mc_j^ld_j^l,\qquad\sum_i i[\tilde\alpha_i+\tilde\beta_i+2(\tilde\gamma_i+\tilde\delta_i)]=d^t.
\end{align}
Denote by $C^\xi_{\textbf{c},g}(\textbf{d},s)$ the following quantity
\begin{equation}
    \sum\mu_s(\mathcal{D})
\end{equation}
where the sum is taken over all floor diagrams of multidegree \textbf{d} and genus $g$ for $P(\textbf{c},\textbf{d})$ and divergence multiplicity vector $\xi'\in(\mathcal{C}^+)^8$ such that $\xi\sim\xi'$.
\begin{remark}\label{rmk-representative}
    Let $\xi=(\alpha,\beta,\gamma,0,\tilde\alpha,\tilde\beta,\tilde\gamma,0)$ and $\xi'=(\alpha,\beta',\gamma,\delta,\tilde\alpha,\tilde\beta',\tilde\gamma,\tilde\delta)$ in $(\mathcal C^+)^8$ such that $\xi\sim\xi'$, then $C^\xi_{\textbf{c},g}(\textbf{d},s)=C^{\xi'}_{\textbf{c},g}(\textbf{d},s)$.
\end{remark}
The next step is to encode these numbers in a map of the form $G_{(\textbf{d}^r;\textbf{d}^l),\textbf{c},g,r}^{n_1,n_2}(\textbf{x},\textbf{y},\textbf{z},\textbf{w})$ defined over the lattice
\begin{equation}
    \tilde\Lambda=\biggl\{(\textbf{x},\textbf{y},\textbf{z},\textbf{w})\in\mathbb Z^{m_1+n_2+m_3+\lfloor\frac{n_2}{2}\rfloor}\bigg|\sum_{i=1}^{m_1}x_i+\sum_{k=1}^{m_3}2z_i+\sum_{j=1}^{n_2}y_j+\sum_{t=1}^{\lfloor\frac{n_2}{2}\rfloor}2w_t+\sum_{i=1}^nc_i^rd_i^r-\sum_{j=1}^mc_j^ld_j^l=0\biggr\},
\end{equation}
where the total number of entries of the vectors $\textbf{x},\textbf{z}$ is equal to $n_1$, i.e. $m_1+2m_3=n_1$, and the total number of nonzero entries of the vectors $\textbf{y},\textbf{w}$ is equal to $n_2$. We associate a vector of sequences $(\alpha,\beta,\gamma,\delta,\tilde\alpha,\tilde\beta,\tilde\gamma,\tilde\delta)$ to $(\textbf{x},\textbf{y},\textbf{z},\textbf{w})\in\tilde\Lambda$ in the following way
\begin{multicols}{2}
    \begin{itemize}
        \item $\alpha_i$ is the number of entries $x_j=-i$.
        \item $\beta_i$ is the number of entries $y_j=-i$.
        \item $\tilde\alpha_i$ is the number of entries $x_j=i$.
        \item $\tilde\beta_i$ is the number of entries $y_j=i$.
        \item $\gamma_i$ is the number of entries $z_j=-i$.
        \item $\delta_i$ is the number of entries $w_j=-i$.
        \item $\tilde\gamma_i$ is the number of entries $z_j=i$.
        \item $\tilde\delta_i$ is the number of entries $w_j=i$.
    \end{itemize}
\end{multicols}
Note that $\xi=(\alpha,\beta,\gamma,\delta,\tilde\alpha,\tilde\beta,\tilde\gamma,\tilde\delta)$ and $(\textbf{x},\textbf{y},\textbf{z},\textbf{w})$ determine each other up to permutations of the entries of the latter, in this case we write $\xi\vdash(\textbf{x},\textbf{y},\textbf{z},\textbf{w})$. Moreover the equivalence relation defined on $(\mathcal C^+)^8$ induces an equivalence relation on the lattice: let $\xi,\xi'\in(\mathcal C^+)^8$ such that $\xi\vdash(\textbf{x}_\xi,\textbf{y}_\xi,\textbf{z}_\xi,\textbf{w}_\xi)$ and $\xi'\vdash(\textbf{x}_{\xi'},\textbf{y}_{\xi'},\textbf{z}_{\xi'},\textbf{w}_{\xi'})$, then $(\textbf{x}_\xi,\textbf{y}_\xi,\textbf{z}_\xi,\textbf{w}_\xi)\sim_{\tilde\Lambda}(\textbf{x}_{\xi'},\textbf{y}_{\xi'},\textbf{z}_{\xi'},\textbf{w}_{\xi'})$ if and only if $\xi\sim\xi'$. By \cref{rmk-representative}, we have that any $(\textbf{x}_{\xi'},\textbf{y}_{\xi'},\textbf{z}_{\xi'},\textbf{w}_{\xi'})$ is equivalent to $(\textbf{x}_\xi,\textbf{y}_\xi,\textbf{z}_\xi,0)$ for some $\xi=(\alpha,\beta,\gamma,0,\tilde\alpha,\tilde\beta,\tilde\gamma,0)$ and $\xi'=(\alpha,\beta',\gamma,\delta',\tilde\alpha,\tilde\beta',\tilde\gamma,\tilde\delta')$ such that $\xi\sim\xi'$. In particular, if $\xi=(\alpha,\beta,\gamma,\delta,\tilde\alpha,\tilde\beta,\tilde\gamma,\tilde\delta)$ is a divergence multiplicity vector, then
\begin{equation}
    G_{(\textbf{d}^r;\textbf{d}^l),\textbf{c},g,s}^{m_1,n_2,m_3}(\textbf{x},\textbf{y},\textbf{z},\textbf{w})=C^\xi_{\textbf{c},g}(\textbf{d},s),
\end{equation}
where $\xi\vdash(\textbf{x},\textbf{y},\textbf{z},\textbf{w})$.
\begin{remark}
    The map $G_{(\textbf{d}^r;\textbf{d}^l),\textbf{c},g,s}^{m_1,n_2,m_3}(\textbf{x},\textbf{y},\textbf{z},\textbf{w})$ does not depend on the class of $(\textbf{x},\textbf{y},\textbf{z},\textbf{w})$ with respect to the equivalence relation $\sim_{\tilde\Lambda}$. In particular,
    \begin{equation}
        G_{(\textbf{d}^r;\textbf{d}^l),\textbf{c},g,s}^{m_1,n_2,m_3}(\textbf{x},\textbf{y},\textbf{z},\textbf{w})=G_{(\textbf{d}^r;\textbf{d}^l),\textbf{c},g,s}^{m_1,n_2,m_3}(\textbf{x},\textbf{y}',\textbf{z},0)\qquad (\textbf{x},\textbf{y},\textbf{z},\textbf{w})\sim_{\tilde\Lambda}(\textbf{x},\textbf{y}',\textbf{z},0).
    \end{equation}
    This means that the map naturally descends on the quotient
    \begin{equation}
        \tilde\Lambda/\sim_{\tilde\Lambda}=\biggl\{(\textbf{x},\textbf{y},\textbf{z})\in\mathbb Z^{m_1+m_3+n_2}\bigg|\sum_{i=1}^{m_1}x_i+\sum_{k=1}^{m_3}2z_i+\sum_{j=1}^{n_2}y_j+\sum_{i=1}^nc_i^rd_i^r-\sum_{j=1}^mc_j^ld_j^l=0\biggr\}.
    \end{equation}
\end{remark}
Consider the following hyperplane arrangement in $\tilde\Lambda$
\begin{equation}
    \sum_{i\in S}x_i+2\sum_{j\in\tilde S}z_j+\sum_{k\in T}y_k+2\sum_{t\in\tilde T}w_t+\sum_{i=1}^nq_ic_i^r-\sum_{j=1}^mp_jc_j^l=0
\end{equation}
\begin{equation}
    y_i-y_j=0\qquad1\leq i<j\leq n_2
\end{equation}
\begin{equation}
    w_i-w_j=0\qquad1\leq i<j\leq\bigg\lfloor\frac{n_2}{2}\bigg\rfloor
\end{equation}
where $S\subseteq[m_1]$, $\tilde S\subseteq[m_3]$, $T\subseteq[n_2]$, $\tilde T\subseteq[\lfloor\frac{n_2}{2}\rfloor]$, $0\leq q_i\leq d_i^r$ and $0\leq p_j\leq d_j^l$ for all $i\in[n]$ and $j\in[m]$. Call this hyperplane arrangement $\mathcal K^{m_1,n_2,m_3}(\textbf{c})$. The image of $\mathcal K^{m_1,n_2,m_3}(\textbf{c})$ in the quotient $\tilde\Lambda/\sim_{\tilde\Lambda}$ gets rid of the variable vector \textbf{w}. We denote the hyperplane arrangement in the quotient as $\mathcal K^{m_1,n_2,m_3}(\textbf{c})/\sim_{\tilde\Lambda}$.\\
Fix $g,n_1,n_2\in\mathbb Z_{\geq0}$, $\textbf{c}=(\textbf{c}^r;\textbf{c}^l)\in\mathbb Z^{n+m}$ such that $c_1^r>\dots>c_n^r$ and $c_1^l<\dots<c_m^l$, $(\textbf{d}^r;\textbf{d}^l)\in(\mathbb Z_{>0})^{n+m}$ and $a=\sum d_i^r=\sum d_j^l$. Let $(\mathcal D,m)$ be an $s$-real floor diagram and denote by $(\tilde{\mathcal D},m)$ the graph obtained from $\mathcal D$ by removing all the weights, but such that the underlying graph $\tilde{\mathcal D}$ inherit the partition $V=L\cup C\cup R$ of the vertices, the ordering of $L,C,R$ and the coloring of the vertices. We call $\mathcal G$ the collection of such pairs $(H,m)$ that contribute to $G_{(\textbf{d}^r;\textbf{d}^l),\textbf{c},g,s}^{m_1,n_2,m_3}(\textbf{x},\textbf{y},\textbf{z},\textbf{w})$. $\mathcal G$ is finite and depends only on $g,a$ and $n_1+n_2$. Let $(H,m)\in\mathcal G$ and define the set $W_{(H,m),\textbf{c},r-l,s}(\textbf{x},\textbf{y},\textbf{z},\textbf{w})$ of weights $w:E(H)\to\mathbb N$ for which the resulting weighted graph is an $s$-real floor diagram for $P(\textbf{c},\textbf{d})$. By construction, the obtained floor diagram has genus $g$ and multidegree $\textbf{d}$. Let us define the following set
\begin{equation}
    \mathsf P_{\mathcal I(H,m,s)}=\left\{w\in\mathbb R^{E(H)}\colon
    \begin{aligned}
        &\bullet w(e)>0\text{ for all } e\in E(H)\\
        &\bullet w(e)\equiv1\bmod 2\text{ for all }e\in E(H)\setminus\mathcal{I}(H,m,s)\\
        &\bullet w(e)=w(\tilde e) \text{ for all $s$-pairs } \{e,\tilde e\}\subseteq\mathcal I(H,m,s)\\
    \end{aligned} 
    \right\}
\end{equation}
and consider the quasipolynomial function $\pi_{\mathcal{I}(H,m,s)}:\mathbb R^{E(H)}\to\mathbb R$ defined by
\begin{equation}
    \pi_{\mathcal I(H,m,s)}(w)=
    \begin{cases}
        \displaystyle\prod w(e) &\text{if }w\in\mathsf P_{\mathcal I(H,m,s)}\\
        0 &\text{otherwise}
    \end{cases}
\end{equation}
where the product runs over all \textbf{internal edges} $e\not\in m\{|\textbf{x}^+|+|\textbf{z}^+|+2s+1,\dots,n\}$, with $|\textbf{x}^+|,|\textbf{z}^+|$ the number of positive entries in $\textbf{x},\textbf{z}$ and $n=n_1+n_2+2a+g-1$.
\begin{remark}\label{rmk-nonzero_contributions}
    Since every graph $H$ in $\mathcal G$ comes with a marking $m$, there exists only one vector $(\textbf{x},\textbf{y}',\textbf{z},\textbf{w}')\sim(\textbf{x},\textbf{y},\textbf{z},\textbf{w})$ such that $W_{(H,m),\textbf{c},r-l,s}(\textbf{x},\textbf{y}',\textbf{z},\textbf{w}')\neq\emptyset$. Indeed, the marking $m$ determines vertices and edges in $\mathcal I(H,m,s)$. The divergence of the vertices in the imaginary part of the central block must agree with $\textbf{w}'$.
\end{remark}
\begin{theorem}\label{thm-rpqp}
    Let $(\textbf{d}^r;\textbf{d}^l)\in\mathbb Z^{n+m}$ be a vector with positive integer coordinates and $g\geq0$, $n_1,n_2>0$ and $\textbf{c}=(\textbf{c}^r;\textbf{c}^l)\in\mathbb Z^{n+m}$ such that $c_1^r>\dots>c_n^r$ and $c_1^l<\dots<c_m^l$. The map $G_{(\textbf{d}^r;\textbf{d}^l),\textbf{c},g,s}^{m_1,n_2,m_3}(\textbf{x},\textbf{y},\textbf{z},\textbf{w})$ is piecewise quasipolynomial in each chamber of $\tilde\Lambda\setminus\mathcal K^{m_1,n_2,m_3}(\textbf{c})$.
\end{theorem}
\begin{proof}[\textit{Proof. }]
    Define the map
    \begin{equation}\label{eq-map1}
        G_{(H,m),\textbf{c},r-l,s}(\textbf{x},\textbf{y},\textbf{z},\textbf{w})=\sum\sum\prod_{i=1}^{m_3} z_i\pi_{\mathcal{I}(H,m,s)}(w),
    \end{equation}
    where the first sum runs over all the vectors $(\textbf{x},\textbf{y}',\textbf{z},\textbf{w}')\sim_{\tilde\Lambda}(\textbf{x},\textbf{y},\textbf{z},\textbf{w})$ and the second sum runs over all $w\in W_{(H,m),\textbf{c},r-l,s}(\textbf{x},\textbf{y}',\textbf{z},\textbf{w}')$. By \cref{rmk-nonzero_contributions}, for each graph $H$ and marking $m$ only one term in the outer sum provides a possibly non-zero contribution. The multiplication by the product of the entries of the vector $\mathbf{z}$ is an expedient to ensure that the quasipolynomial weight has the same form as the function $\pi_Y$ in \cref{cor-EhrhartY_qp}. Then we can write
    \small
    \begin{equation}
        G_{(\textbf{d}^r;\textbf{d}^l),\textbf{c},g,s}^{m_1,n_2,m_3}(\textbf{x},\textbf{y},\textbf{z},\textbf{w})=\frac{1}{\beta!\delta!\tilde\beta!\tilde\delta!}\sum_{(H,m)\in\mathcal G}(-1)^{\frac{|BV(H)\cap\mathcal I(H,m,s)|}{2}}\sum_{(r,l)}\sum_{\substack{(\sigma,\tau)\in\\ S_{n_2}\times S_{\lfloor\frac{n_2}{2}\rfloor}}}\frac{G_{(H,m),\textbf{c},r-l,s}(\textbf{x},\sigma(\textbf{y}),\textbf{z},\tau(\textbf{w}))}{\prod z_i}.
    \end{equation}
    \normalsize
    The idea is the same as in the proof of \cref{thm-quasipolynomiality}. The internal sum in \cref{eq-map1} equals the weighted vector partition function $\mathcal P_{H,\pi_{\mathcal I(H,m,s)}}(\textbf{k})$, which is by \cref{cor-EhrhartY_qp} quasipolynomial relative to the chambers of the discriminant arrangement $\{\textbf{k}\in\mathbb R^{V(H)}|\sum\textbf{k}(v)=0\}$, when restricted to the subspace $H_{r-l}(\textbf{x},\textbf{y}',\textbf{z},\textbf{w}')$ determined by the equations
    \begin{align}
        &\textbf{k}(w_i)=x_i,\qquad\textbf{k}(w_j)=y'_j,\qquad\textbf{k}(b_i)=r_i-l_i\text{ for all black }b_i\\
        &\textbf{k}(\tilde w_h)=\textbf{k}(\tilde w_{h+1})=z_i,\qquad\textbf{k}(\tilde w_e)=\textbf{k}(\tilde w_{e+1})=w'_e.
    \end{align}
    Therefore, the internal sum is piecewise quasipolynomial relative to the chambers of the discriminant arrangement in $H_{r-l}(\textbf{x},\textbf{y}',\textbf{z},\textbf{w}')$, denoted by $S_{r-l}(\textbf{x},\textbf{y}',\textbf{z},\textbf{w}')$. The map $G_{(H,m),\textbf{c},r-l,s}(\textbf{x},\textbf{y},\textbf{z},\textbf{w})$ is then piecewise quasipolynomial relative to the chambers of the common refinement $\displaystyle\bigcup S_{r-l}(\textbf{x},\textbf{y}',\textbf{z},\textbf{w}')$. Finally, by following the same reasoning in the proof of \cref{thm-quasipolynomiality}, the map $G_{(\textbf{d}^r;\textbf{d}^l),\textbf{c},g,s}^{m_1,n_2,m_3}(\textbf{x},\textbf{y},\textbf{z},\textbf{w})$ is piecewise quasipolynomial relative to the chambers of 
    \begin{equation}
        \bigcup_{(r,l)}\bigcup_{(\textbf{x},\textbf{y}',\textbf{z},\textbf{w}')\sim_{\tilde\Lambda}(\textbf{x},\textbf{y},\textbf{z},\textbf{w})} S_{r-l}(\textbf{x},\textbf{y}',\textbf{z},\textbf{w}').\tag*{\qedhere}
    \end{equation}
    
\end{proof}
\begin{corollary}
    The map $\tilde G_{(\textbf{d}^r;\textbf{d}^l),\textbf{c},g,s}^{m_1,n_2,m_3}(\textbf{x},\textbf{y},\textbf{z})$ defined on the quotient $\tilde\Lambda/\sim_{\tilde\Lambda}$ is piecewise quasipolynomial relative to the chambers of $(\tilde\Lambda/\sim_{\tilde\Lambda})\setminus(\mathcal K^{m_1,n_2,m_3}(\textbf{c})/\sim_{\tilde\Lambda})$.
\end{corollary}
\begin{remark}
    In the case $s=0$ and $\textbf{z}=0$, the map $G_{(\textbf{d}^r;\textbf{d}^l),\textbf{c},g,0}^{n_1,n_2,0}(\textbf{x},\textbf{y},0,0)$ coincides with the map $G_{(\textbf{d}^r;\textbf{d}^l),\textbf{c},g}^{n_1,n_2}(\textbf{x},\textbf{y})$. Therefore \cref{thm-rpqp} can be seen as a generalization of \cref{thm-quasipolynomiality}.
\end{remark}

\subsection{Example}

In this section, we provide an example to illustrate the result of \cref{thm-rpqp}. We consider the same data of \cref{subsec:example1} except for $g=0$ and the parameter $s=1$. Here, the domain $\tilde\Lambda$ has the following form
\begin{equation}
    \tilde\Lambda=\{(x_1,y_1,y_2,w_1)\in\mathbb Z^4|x_1+y_1+y_2+2w_1+k=0\}.
\end{equation}
First of all, note that we have $(x_1,y_1,y_2,0)\sim_{\tilde\Lambda}(x_1,0,0,w_1)$ if and only if $y_1=y_2$ and $w_1=y_1$. We depict the hyperplane arrangement in \cref{fig:hyperplane_arrangement2}. 

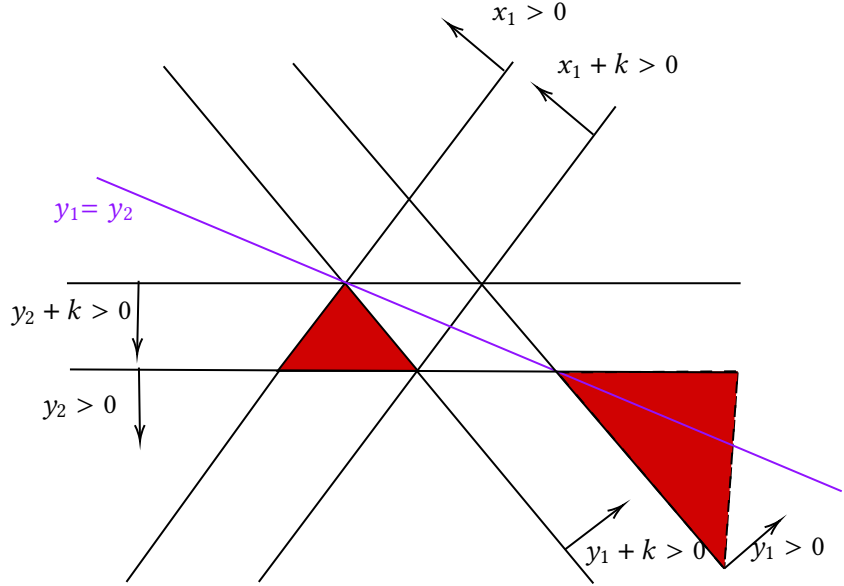
\begin{figure}
    \centering       

    \tikzset{every picture/.style={line width=0.75pt}} %set default line width to 0.75pt        

    \begin{tikzpicture}[x=0.75pt,y=0.75pt,yscale=-0.75,xscale=0.75]
    %uncomment if require: \path (0,708); %set diagram left start at 0, and has height of 708

        %Straight Lines [id:da1392665430229072] 
        \draw    (390,452) -- (491.71,311.3) -- (650,102) ;
        %Straight Lines [id:da6693041887501123] 
        \draw    (479,452) -- (719,132) ;
        %Straight Lines [id:da647155129306513] 
        \draw    (350,251) -- (803,251) ;
        %Straight Lines [id:da4946947233704948] 
        \draw    (704,453) -- (415,105) ;
        %Straight Lines [id:da2994470874358284] 
        \draw    (644,108) -- (607.53,77.29) ;
        \draw [shift={(606,76)}, rotate = 40.1] [color={rgb, 255:red, 0; green, 0; blue, 0 }  ][line width=0.75]    (10.93,-3.29) .. controls (6.95,-1.4) and (3.31,-0.3) .. (0,0) .. controls (3.31,0.3) and (6.95,1.4) .. (10.93,3.29)   ;
        %Straight Lines [id:da587505475862875] 
        \draw    (704,151) -- (667.53,120.29) ;
        \draw [shift={(666,119)}, rotate = 40.1] [color={rgb, 255:red, 0; green, 0; blue, 0 }  ][line width=0.75]    (10.93,-3.29) .. controls (6.95,-1.4) and (3.31,-0.3) .. (0,0) .. controls (3.31,0.3) and (6.95,1.4) .. (10.93,3.29)   ;
        %Straight Lines [id:da7396797328116171] 
        \draw    (792,443) -- (828.08,411.83) ;
        \draw [shift={(829.6,410.53)}, rotate = 139.18] [color={rgb, 255:red, 0; green, 0; blue, 0 }  ][line width=0.75]    (10.93,-3.29) .. controls (6.95,-1.4) and (3.31,-0.3) .. (0,0) .. controls (3.31,0.3) and (6.95,1.4) .. (10.93,3.29)   ;
        %Straight Lines [id:da5193072970569319] 
        \draw    (684.76,430.28) -- (722.64,401.33) ;
        \draw [shift={(724.23,400.12)}, rotate = 142.62] [color={rgb, 255:red, 0; green, 0; blue, 0 }  ][line width=0.75]    (10.93,-3.29) .. controls (6.95,-1.4) and (3.31,-0.3) .. (0,0) .. controls (3.31,0.3) and (6.95,1.4) .. (10.93,3.29)   ;
        %Straight Lines [id:da3790203465520465] 
        \draw    (398.2,249.35) -- (397.22,297.02) ;
        \draw [shift={(397.18,299.02)}, rotate = 271.18] [color={rgb, 255:red, 0; green, 0; blue, 0 }  ][line width=0.75]    (10.93,-3.29) .. controls (6.95,-1.4) and (3.31,-0.3) .. (0,0) .. controls (3.31,0.3) and (6.95,1.4) .. (10.93,3.29)   ;
        %Straight Lines [id:da5492089170150714] 
        \draw    (398.39,307.74) -- (399.09,355.42) ;
        \draw [shift={(399.12,357.42)}, rotate = 269.16] [color={rgb, 255:red, 0; green, 0; blue, 0 }  ][line width=0.75]    (10.93,-3.29) .. controls (6.95,-1.4) and (3.31,-0.3) .. (0,0) .. controls (3.31,0.3) and (6.95,1.4) .. (10.93,3.29)   ;
        %Shape: Triangle [id:dp8024624317878942] 
        \draw  [fill={rgb, 255:red, 208; green, 2; blue, 2 }  ,fill opacity=1 ] (537,251) -- (586,310) -- (491.71,310) -- cycle ;
        %Straight Lines [id:da403828750718922] 
        \draw  [dash pattern={on 4.5pt off 4.5pt}]  (801,311) -- (792,443) ;
        %Shape: Triangle [id:dp7152191171784666] 
        \draw  [fill={rgb, 255:red, 208; green, 2; blue, 2 }  ,fill opacity=1 ][dash pattern={on 4.5pt off 4.5pt}][line width=0.75]  (679.87,311.06) -- (801.04,310.26) -- (792.01,443) -- cycle ;
        %Straight Lines [id:da0317587792891425] 
        \draw [color={rgb, 255:red, 144; green, 19; blue, 254 }  ,draw opacity=1 ]   (370,180) -- (871,391) ;
        %Straight Lines [id:da2770936606819433] 
        \draw    (792.01,443) -- (502.01,103) ;
        %Straight Lines [id:da11185152970017498] 
        \draw    (352,309) -- (801,311) ;

        % Text Node
        \draw (634,60) node [anchor=north west][inner sep=0.75pt]    {$x_{1}  >0$};
        % Text Node
        \draw (678,94) node [anchor=north west][inner sep=0.75pt]    {$x_{1} +k >0$};
        % Text Node
        \draw (806.62,419.76) node [anchor=north west][inner sep=0.75pt]    {$y_{1}  >0$};
        % Text Node
        \draw (695.69,420.86) node [anchor=north west][inner sep=0.75pt]    {$y_{1} +k >0$};
        % Text Node
        \draw (309.19,259.62) node [anchor=north west][inner sep=0.75pt]    {$y_{2} +k >0$};
        % Text Node
        \draw (330.4,324.23) node [anchor=north west][inner sep=0.75pt]    {$y_{2}  >0$};
        % Text Node
        \draw (338,195) node [anchor=north west][inner sep=0.75pt]    {$\textcolor[rgb]{0.56,0.07,1}{y}\textcolor[rgb]{0.56,0.07,1}{_{1}}\textcolor[rgb]{0.56,0.07,1}{=y}\textcolor[rgb]{0.56,0.07,1}{_{2}}$};

    \end{tikzpicture}
    
    \caption{Hyperplane arrangement in $\tilde\Lambda$.}
    \label{fig:hyperplane_arrangement2}
\end{figure}

The chambers in red are the chambers in which we compute the map $G^{1,2}_{(\textbf{d}^r;\textbf{d}^l),\textbf{c},0,1}(x_1,y_1,y_2,0)$. We denote by $\mathbf{C}_1$ the chamber given by the inequalities $-k<x_1<0$, $y_1<-k$ and $-k<y_2<0$ and $\mathbf{C}_2$ the chamber given by the inequalities $x_1<-k$ and $y_1,y_2>0$. The choice of the chambers is not random: in the chamber $\mathbf{C}_1$, $y_1=y_2$ cannot happen, therefore the floor diagrams contributing to $G^{1,2}_{(\textbf{d}^r;\textbf{d}^l),\textbf{c},0,1}(x_1,y_1,y_2,0)$ can only have divergence sequence $(x_1,y_1,y_2,0)$, while in $\mathbf{C}_2$, on the line $y_1=y_2$, we have non-zero contribution from floor diagrams having divergence sequence $(x_1,y_1,y_1,0)$ and $(x_1,0,0,y_1)$.

Recall that $r_1-l=(k,0)$ and $r_2-l=(0,k)$. We compute first the map\\ 
$G^{1,2}_{(\textbf{d}^r;\textbf{d}^l),\textbf{c},0,1}(x_1,y_1,y_2,0)$ in the chamber $\mathbf{C}_1$. We list all the $1$-real floor diagrams satisfying the inequalities of the chamber, with divergence sequence $(x_1,y_1,y_2,0)$ and black vertices having divergence $r_1-l$ and $r_2-l$ in \cref{table3} and \cref{table4} respectively. Note that in \cref{table3} we provide markings for some floor diagrams.

\begin{table}
        \centering
        
\begin{tabular}{|p{0.02\textwidth}|p{0.29\textwidth}|p{0.29\textwidth}|p{0.29\textwidth}|}
        \hline 
        & \begin{center}
            $\displaystyle 1$
        \end{center}
        & \begin{center}
            $\displaystyle 2$
        \end{center}
        & \begin{center}
            $\displaystyle 3$
        \end{center}
        \\
\hline
    \end{tabular}

    \caption{$1$-real floor diagrams with divergence sequence $r_1-l$ contributing to $G^{1,2}_{(\textbf{d}^r;\textbf{d}^l),\textbf{c},0,1}(x_1,y_1,y_2,0)$ in $\mathbf{C}_1$.}
    \label{table3}
        
\end{table}

\begin{table}
    \centering
        
    \begin{tabular}{|p{0.02\textwidth}|p{0.29\textwidth}|p{0.29\textwidth}|p{0.29\textwidth}|}
        \hline 
        & \begin{center}
            $\displaystyle 1$
        \end{center}
        & \begin{center}
            $\displaystyle 2$
        \end{center}
        & \begin{center}
            $\displaystyle 3$
        \end{center}
        \\
\hline
    \end{tabular}

    \caption{$1$-real floor diagrams with divergence sequence $r_2-l$ contributing to $G^{1,2}_{(\textbf{d}^r;\textbf{d}^l),\textbf{c},0,1}(x_1,y_1,y_2,0)$ in $\mathbf{C}_1$.}
    \label{table4}
        
\end{table}\vspace{\baselineskip}    
As an example, we compute the multiplicity of the floor diagram $C1$ and then we provide the expression of the map $G^{1,2}_{(\textbf{d}^r;\textbf{d}^l),\textbf{c},0,1}(x_1,y_1,y_2,0)$. The only $1$-pair is $\{1,2\}$. Therefore, both floor diagrams in \cref{fig:example_multiplicity} have empty imaginary part and the function $\rho_{\mathcal D,m}$ is the identity in each case, so they are $1$-real floor diagram having multiplicity $-y_1$ and $-y_2$ respectively, as long as $x_1,y_1,y_2,k$ are odd. We distinguish, as in \cref{subsec:example1}, two cases.\\
If $k\geq0$ is \textbf{even}, $x_1$ is even as long as $y_1$ and $y_2$ are odd. Therefore, in this case $G^{1,2}_{(\textbf{d}^r;\textbf{d}^l),\textbf{c},0,1}(x_1,y_1,y_2,0)=0$.\\
If $k\geq0$ is \textbf{odd}, the graphs contributing non-zero to $G^{1,2}_{(\textbf{d}^r;\textbf{d}^l),\textbf{c},0,1}(x_1,y_1,y_2,0)$ are $A1$, $B2$, $B3$, $C1$, $C2$, $C3$, $D1$, $D3$, $E1$, $E2$, $E3$ in \cref{table3}. The graphs in \cref{table4} have multiplicity zero since the black-black edge has even weight. So we get
\begin{equation}
    2k-6x_1-7y_1-7y_2\qquad\text{ if }x_1,y_1,y_2\text{ are odd.}
\end{equation}
Putting all together
\begin{equation}
    G^{1,2}_{(\textbf{d}^r;\textbf{d}^l),0,1}(x_1,y_1,y_2,0,k)=
    \begin{cases}
        2k-6x_1-7y_1-7y_2 &\text{if }x_1,y_1,y_2,k\text{ are odd}\\
        0 &\text{otherwise}
    \end{cases}
\end{equation}
Note that for each graph we have two non-zero contributions since we have to interchange $y_1$ and $y_2$.

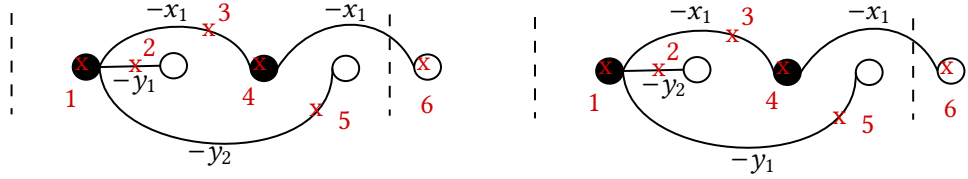
\begin{figure}
    \centering
    \tikzset{every picture/.style={line width=0.75pt}} %set default line width to 0.75pt        

    \begin{tikzpicture}[x=0.75pt,y=0.75pt,yscale=-1,xscale=1]
    %uncomment if require: \path (0,155); %set diagram left start at 0, and has height of 155

        %Shape: Ellipse [id:dp030960130868250224] 
        \draw   (219.89,48.27) .. controls (219.89,44.74) and (222.92,41.88) .. (226.65,41.88) .. controls (230.39,41.88) and (233.41,44.74) .. (233.41,48.27) .. controls (233.41,51.8) and (230.39,54.66) .. (226.65,54.66) .. controls (222.92,54.66) and (219.89,51.8) .. (219.89,48.27) -- cycle ;
        %Shape: Ellipse [id:dp21087651585991962] 
        \draw  [fill={rgb, 255:red, 0; green, 0; blue, 0 }  ,fill opacity=1 ] (47.67,48.27) .. controls (47.67,44.74) and (50.7,41.88) .. (54.43,41.88) .. controls (58.17,41.88) and (61.19,44.74) .. (61.19,48.27) .. controls (61.19,51.8) and (58.17,54.66) .. (54.43,54.66) .. controls (50.7,54.66) and (47.67,51.8) .. (47.67,48.27) -- cycle ;
        %Shape: Ellipse [id:dp9642104826769369] 
        \draw   (91.8,47.6) .. controls (91.8,44.07) and (94.82,41.21) .. (98.56,41.21) .. controls (102.29,41.21) and (105.32,44.07) .. (105.32,47.6) .. controls (105.32,51.13) and (102.29,53.99) .. (98.56,53.99) .. controls (94.82,53.99) and (91.8,51.13) .. (91.8,47.6) -- cycle ;
        %Shape: Ellipse [id:dp5200370876322641] 
        \draw  [fill={rgb, 255:red, 0; green, 0; blue, 0 }  ,fill opacity=1 ] (137.34,48.94) .. controls (137.34,45.41) and (140.37,42.55) .. (144.1,42.55) .. controls (147.83,42.55) and (150.86,45.41) .. (150.86,48.94) .. controls (150.86,52.48) and (147.83,55.34) .. (144.1,55.34) .. controls (140.37,55.34) and (137.34,52.48) .. (137.34,48.94) -- cycle ;
        %Shape: Ellipse [id:dp9335506693647481] 
        \draw  [fill={rgb, 255:red, 255; green, 255; blue, 255 }  ,fill opacity=1 ] (178.62,48.94) .. controls (178.62,45.41) and (181.64,42.55) .. (185.38,42.55) .. controls (189.11,42.55) and (192.14,45.41) .. (192.14,48.94) .. controls (192.14,52.48) and (189.11,55.34) .. (185.38,55.34) .. controls (181.64,55.34) and (178.62,52.48) .. (178.62,48.94) -- cycle ;
        %Straight Lines [id:da20452062763667456] 
        \draw  [dash pattern={on 4.5pt off 4.5pt}]  (16.83,21.02) -- (16.83,72.16) ;
        %Straight Lines [id:da13584017398600534] 
        \draw  [dash pattern={on 4.5pt off 4.5pt}]  (207.55,21.69) -- (207.55,72.83) ;
        %Curve Lines [id:da40723115152896805] 
        \draw    (61.19,48.27) .. controls (70.98,22.36) and (129.33,19.67) .. (137.34,48.94) ;
        %Straight Lines [id:da8037467632754793] 
        \draw    (61.19,48.27) -- (91.8,47.6) ;
        %Curve Lines [id:da30734335406358204] 
        \draw    (150.86,48.94) .. controls (169.19,19.67) and (207.62,19.67) .. (219.89,48.27) ;
        %Curve Lines [id:da5332008651818072] 
        \draw    (61.19,48.27) .. controls (63.77,97.06) and (180.48,102.44) .. (178.62,48.94) ;
        %Shape: Ellipse [id:dp7095483781318458] 
        \draw   (483.89,50.27) .. controls (483.89,46.74) and (486.92,43.88) .. (490.65,43.88) .. controls (494.39,43.88) and (497.41,46.74) .. (497.41,50.27) .. controls (497.41,53.8) and (494.39,56.66) .. (490.65,56.66) .. controls (486.92,56.66) and (483.89,53.8) .. (483.89,50.27) -- cycle ;
        %Shape: Ellipse [id:dp46195143898559154] 
        \draw  [fill={rgb, 255:red, 0; green, 0; blue, 0 }  ,fill opacity=1 ] (311.67,50.27) .. controls (311.67,46.74) and (314.7,43.88) .. (318.43,43.88) .. controls (322.17,43.88) and (325.19,46.74) .. (325.19,50.27) .. controls (325.19,53.8) and (322.17,56.66) .. (318.43,56.66) .. controls (314.7,56.66) and (311.67,53.8) .. (311.67,50.27) -- cycle ;
        %Shape: Ellipse [id:dp2651026846053923] 
        \draw   (355.8,49.6) .. controls (355.8,46.07) and (358.82,43.21) .. (362.56,43.21) .. controls (366.29,43.21) and (369.32,46.07) .. (369.32,49.6) .. controls (369.32,53.13) and (366.29,55.99) .. (362.56,55.99) .. controls (358.82,55.99) and (355.8,53.13) .. (355.8,49.6) -- cycle ;
        %Shape: Ellipse [id:dp6524052110538314] 
        \draw  [fill={rgb, 255:red, 0; green, 0; blue, 0 }  ,fill opacity=1 ] (401.34,50.94) .. controls (401.34,47.41) and (404.37,44.55) .. (408.1,44.55) .. controls (411.83,44.55) and (414.86,47.41) .. (414.86,50.94) .. controls (414.86,54.48) and (411.83,57.34) .. (408.1,57.34) .. controls (404.37,57.34) and (401.34,54.48) .. (401.34,50.94) -- cycle ;
        %Shape: Ellipse [id:dp5794812964765192] 
        \draw  [fill={rgb, 255:red, 255; green, 255; blue, 255 }  ,fill opacity=1 ] (442.62,50.94) .. controls (442.62,47.41) and (445.64,44.55) .. (449.38,44.55) .. controls (453.11,44.55) and (456.14,47.41) .. (456.14,50.94) .. controls (456.14,54.48) and (453.11,57.34) .. (449.38,57.34) .. controls (445.64,57.34) and (442.62,54.48) .. (442.62,50.94) -- cycle ;
        %Straight Lines [id:da3970724749278167] 
        \draw  [dash pattern={on 4.5pt off 4.5pt}]  (280.83,23.02) -- (280.83,74.16) ;
        %Straight Lines [id:da42731566569022705] 
        \draw  [dash pattern={on 4.5pt off 4.5pt}]  (471.55,23.69) -- (471.55,74.83) ;
        %Curve Lines [id:da4046670716966445] 
        \draw    (325.19,50.27) .. controls (334.98,24.36) and (393.33,21.67) .. (401.34,50.94) ;
        %Straight Lines [id:da2135694307811844] 
        \draw    (325.19,50.27) -- (355.8,49.6) ;
        %Curve Lines [id:da361799365686154] 
        \draw    (414.86,50.94) .. controls (433.19,21.67) and (471.62,21.67) .. (483.89,50.27) ;
        %Curve Lines [id:da143374241201911] 
        \draw    (325.19,50.27) .. controls (327.77,99.06) and (444.48,104.44) .. (442.62,50.94) ;

        % Text Node
        \draw (47.29,42.09) node [anchor=north west][inner sep=0.75pt]   [align=left] {\textcolor[rgb]{0.82,0.01,0.01}{x}};
        % Text Node
        \draw (74.33,43.09) node [anchor=north west][inner sep=0.75pt]   [align=left] {\textcolor[rgb]{0.82,0.01,0.01}{x}};
        % Text Node
        \draw (111.34,24.94) node [anchor=north west][inner sep=0.75pt]   [align=left] {\textcolor[rgb]{0.82,0.01,0.01}{x}};
        % Text Node
        \draw (136.96,42.09) node [anchor=north west][inner sep=0.75pt]   [align=left] {\textcolor[rgb]{0.82,0.01,0.01}{x}};
        % Text Node
        \draw (165.42,65.28) node [anchor=north west][inner sep=0.75pt]   [align=left] {\textcolor[rgb]{0.82,0.01,0.01}{x}};
        % Text Node
        \draw (219.51,42.09) node [anchor=north west][inner sep=0.75pt]   [align=left] {\textcolor[rgb]{0.82,0.01,0.01}{x}};
        % Text Node
        \draw (41.81,58.11) node [anchor=north west][inner sep=0.75pt]   {$\textcolor[rgb]{0.82,0.01,0.01}{1}$};
        % Text Node
        \draw (81.42,32.43) node [anchor=north west][inner sep=0.75pt]   {$\textcolor[rgb]{0.82,0.01,0.01}{2}$};
        % Text Node
        \draw (119.36,14.66) node [anchor=north west][inner sep=0.75pt]  {$\textcolor[rgb]{0.82,0.01,0.01}{3}$};
        % Text Node
        \draw (131.48,57.1) node [anchor=north west][inner sep=0.75pt]  {$\textcolor[rgb]{0.82,0.01,0.01}{4}$};
        % Text Node
        \draw (179.87,68.87) node [anchor=north west][inner sep=0.75pt]  {$\textcolor[rgb]{0.82,0.01,0.01}{5}$};
        % Text Node
        \draw (221.59,60.8) node [anchor=north west][inner sep=0.75pt]   {$\textcolor[rgb]{0.82,0.01,0.01}{6}$};
        % Text Node
        \draw (66,51) node [anchor=north west][inner sep=0.75pt]    {$-y_{1}$};
        % Text Node
        \draw (104,87) node [anchor=north west][inner sep=0.75pt]    {$-y_{2}$};
        % Text Node
        \draw (173,15) node [anchor=north west][inner sep=0.75pt]    {$-x_{1}$};
        % Text Node
        \draw (83,15) node [anchor=north west][inner sep=0.75pt]    {$-x_{1}$};
        % Text Node
        \draw (311.29,44.09) node [anchor=north west][inner sep=0.75pt]   [align=left] {\textcolor[rgb]{0.82,0.01,0.01}{x}};
        % Text Node
        \draw (338.33,45.09) node [anchor=north west][inner sep=0.75pt]   [align=left] {\textcolor[rgb]{0.82,0.01,0.01}{x}};
        % Text Node
        \draw (375.34,27.94) node [anchor=north west][inner sep=0.75pt]   [align=left] {\textcolor[rgb]{0.82,0.01,0.01}{x}};
        % Text Node
        \draw (400.96,44.09) node [anchor=north west][inner sep=0.75pt]   [align=left] {\textcolor[rgb]{0.82,0.01,0.01}{x}};
        % Text Node
        \draw (429.42,69.28) node [anchor=north west][inner sep=0.75pt]   [align=left] {\textcolor[rgb]{0.82,0.01,0.01}{x}};
        % Text Node
        \draw (483.51,44.09) node [anchor=north west][inner sep=0.75pt]   [align=left] {\textcolor[rgb]{0.82,0.01,0.01}{x}};
        % Text Node
        \draw (305.81,60.11) node [anchor=north west][inner sep=0.75pt]  {$\textcolor[rgb]{0.82,0.01,0.01}{1}$};
        % Text Node
        \draw (345.42,34.43) node [anchor=north west][inner sep=0.75pt]  {$\textcolor[rgb]{0.82,0.01,0.01}{2}$};
        % Text Node
        \draw (383.36,16.66) node [anchor=north west][inner sep=0.75pt]  {$\textcolor[rgb]{0.82,0.01,0.01}{3}$};
        % Text Node
        \draw (395.48,59.1) node [anchor=north west][inner sep=0.75pt]  {$\textcolor[rgb]{0.82,0.01,0.01}{4}$};
        % Text Node
        \draw (443.87,70.87) node [anchor=north west][inner sep=0.75pt]  {$\textcolor[rgb]{0.82,0.01,0.01}{5}$};
        % Text Node
        \draw (485.59,62.8) node [anchor=north west][inner sep=0.75pt]  {$\textcolor[rgb]{0.82,0.01,0.01}{6}$};
        % Text Node
        \draw (378,91) node [anchor=north west][inner sep=0.75pt]    {$-y_{1}$};
        % Text Node
        \draw (333,53) node [anchor=north west][inner sep=0.75pt]    {$-y_{2}$};
        % Text Node
        \draw (437,16) node [anchor=north west][inner sep=0.75pt]    {$-x_{1}$};
        % Text Node
        \draw (347,16) node [anchor=north west][inner sep=0.75pt]    {$-x_{1}$};

    \end{tikzpicture}

    \caption{The floor diagram $C1$ with two different weightings.}
    \label{fig:example_multiplicity}
\end{figure}

Now, we compute the map $G^{1,2}_{(\textbf{d}^r;\textbf{d}^l),\textbf{c},0,1}(x_1,y_1,y_2,0)$ in the chamber $\mathbf{C}_2$. We list all the floor diagrams having black divergence sequence $r_1-l$ in \cref{table5} and $r_2-l$ in \cref{table6}. Note that the floor diagrams $A'1,A'2,B'2,C'1$ contribute to $G^{1,2}_{(\textbf{d}^r;\textbf{d}^l),\textbf{c},0,1}(x_1,y_1,y_2,0)$ only when $y_1=y_2$.\\
If $k\geq0$ is \textbf{even}, we have again $G^{1,2}_{(\textbf{d}^r;\textbf{d}^l),\textbf{c},0,1}(x_1,y_1,y_2,0)=0$.\\
If $k\geq0$ is \textbf{odd}, then $G^{1,2}_{(\textbf{d}^r;\textbf{d}^l),\textbf{c},0,1}(x_1,y_1,y_2,0)$ has the following form
\begin{equation}
    G^{1,2}_{(\textbf{d}^r;\textbf{d}^l),\textbf{c},0,1}(x_1,y_1,y_2,0)=
    \begin{cases}
        2k+y_1+y_2 &\text{if }y_1\neq y_2\text{ and }x_1,y_1,y_2\text{ odd}\\
        3y_1^2 &\text{if }y_1=y_2,\>x_1\text{ odd, }y_1\text{ even}\\
        3y_1^2+2y_1+2k &\text{if }y_1=y_2,\>x_1,y_1\text{ odd}\\
        0 &\text{otherwise}
    \end{cases}
\end{equation}

In particular, the first polynomial piece comes from the graphs $A3$ and $C'2$, while the second and third polynomial pieces come from the graphs $A1$, $A2$, $B1$ and $A1$, $A2$, $A3$, $C'1$, $C'2$ respectively. As an example we compute the multiplicity of $A'2$ in \cref{fig:computation}: the only $1$-pair is $\{1,2\}$ and it belongs to the imaginary part since $1$ and $2$ are not adjacent. It follows that $\rho_{\mathcal D,m}$ exchanges $1$ and $2$, therefore the floor diagram is a $1$-real floor diagram if and only if $y_1=y_2$ and in this case its multiplicity is $y_1^2$. 

\begin{figure}[H]
    \centering
    \tikzset{every picture/.style={line width=0.75pt}} %set default line width to 0.75pt        

    \begin{tikzpicture}[x=1pt,y=1pt,yscale=-1,xscale=1]
    %uncomment if require: \path (0,195); %set diagram left start at 0, and has height of 195

        %Shape: Ellipse [id:dp4220938223895735] 
        \draw   (362.25,151.3) .. controls (362.25,147.89) and (364.99,145.13) .. (368.37,145.13) .. controls (371.76,145.13) and (374.5,147.89) .. (374.5,151.3) .. controls (374.5,154.71) and (371.76,157.47) .. (368.37,157.47) .. controls (364.99,157.47) and (362.25,154.71) .. (362.25,151.3) -- cycle ;
        %Shape: Ellipse [id:dp7726897794470692] 
        \draw  [fill={rgb, 255:red, 255; green, 255; blue, 255 }  ,fill opacity=1 ] (206.17,150.65) .. controls (206.17,147.24) and (208.92,144.48) .. (212.3,144.48) .. controls (215.68,144.48) and (218.43,147.24) .. (218.43,150.65) .. controls (218.43,154.06) and (215.68,156.82) .. (212.3,156.82) .. controls (208.92,156.82) and (206.17,154.06) .. (206.17,150.65) -- cycle ;
        %Shape: Ellipse [id:dp5282240132679368] 
        \draw   (250.03,151.3) .. controls (250.03,147.89) and (252.77,145.13) .. (256.16,145.13) .. controls (259.54,145.13) and (262.28,147.89) .. (262.28,151.3) .. controls (262.28,154.71) and (259.54,157.47) .. (256.16,157.47) .. controls (252.77,157.47) and (250.03,154.71) .. (250.03,151.3) -- cycle ;
        %Shape: Ellipse [id:dp08314303669558132] 
        \draw  [fill={rgb, 255:red, 0; green, 0; blue, 0 }  ,fill opacity=1 ] (287.43,151.3) .. controls (287.43,147.89) and (290.18,145.13) .. (293.56,145.13) .. controls (296.94,145.13) and (299.69,147.89) .. (299.69,151.3) .. controls (299.69,154.71) and (296.94,157.47) .. (293.56,157.47) .. controls (290.18,157.47) and (287.43,154.71) .. (287.43,151.3) -- cycle ;
        %Shape: Ellipse [id:dp3476707690603458] 
        \draw  [fill={rgb, 255:red, 0; green, 0; blue, 0 }  ,fill opacity=1 ] (324.84,151.3) .. controls (324.84,147.89) and (327.58,145.13) .. (330.97,145.13) .. controls (334.35,145.13) and (337.09,147.89) .. (337.09,151.3) .. controls (337.09,154.71) and (334.35,157.47) .. (330.97,157.47) .. controls (327.58,157.47) and (324.84,154.71) .. (324.84,151.3) -- cycle ;
        %Straight Lines [id:da8486319284274382] 
        \draw  [dash pattern={on 4.5pt off 4.5pt}]  (189.83,125.65) -- (189.83,175) ;
        %Straight Lines [id:da34547756254043194] 
        \draw  [dash pattern={on 4.5pt off 4.5pt}]  (351.06,125) -- (351.06,174.35) ;
        %Curve Lines [id:da03317112159016644] 
        \draw    (218.43,150.65) .. controls (236.28,121.75) and (314.96,123.05) .. (324.84,151.3) ;
        %Straight Lines [id:da6850010153443786] 
        \draw    (299.69,151.3) -- (324.84,151.3) ;
        %Straight Lines [id:da3487167350921443] 
        \draw    (337.09,151.3) -- (362.25,151.3) ;
        %Curve Lines [id:da40971608412310545] 
        \draw    (262.28,151.3) .. controls (267.24,182.79) and (317.54,182.79) .. (324.84,151.3) ;

        % Text Node
        \draw (256,166) node [anchor=north west][inner sep=0.75pt]    {$y_{1}$};
        % Text Node
        \draw (236,120) node [anchor=north west][inner sep=0.75pt]    {$y_{2}$};
        % Text Node
        \draw (305,140) node [anchor=north west][inner sep=0.75pt]    {$k$};
        % Text Node
        \draw (332,140) node [anchor=north west][inner sep=0.75pt]    {$-x_{1}$};
        % Text Node
        \draw (226.29,136.09) node [anchor=north west][inner sep=0.75pt]   [align=left] {\textcolor[rgb]{0.82,0.01,0.01}{x}};
        % Text Node
        \draw (281.29,170.09) node [anchor=north west][inner sep=0.75pt]   [align=left] {\textcolor[rgb]{0.82,0.01,0.01}{x}};
        % Text Node
        \draw (290.29,148.09) node [anchor=north west][inner sep=0.75pt]   [align=left] {\textcolor[rgb]{0.82,0.01,0.01}{x}};
        % Text Node
        \draw (307.29,148.09) node [anchor=north west][inner sep=0.75pt]   [align=left] {\textcolor[rgb]{0.82,0.01,0.01}{x}};
        % Text Node
        \draw (328.29,148.09) node [anchor=north west][inner sep=0.75pt]   [align=left] {\textcolor[rgb]{0.82,0.01,0.01}{x}};
        % Text Node
        \draw (365.29,148.09) node [anchor=north west][inner sep=0.75pt]   [align=left] {\textcolor[rgb]{0.82,0.01,0.01}{x}};
        % Text Node
        \draw (226.29,128.11) node [anchor=north west][inner sep=0.75pt]  [font=\normalsize]  {$\textcolor[rgb]{0.82,0.01,0.01}{1}$};
        % Text Node
        \draw (281.29,176.43) node [anchor=north west][inner sep=0.75pt]  [font=\normalsize]  {$\textcolor[rgb]{0.82,0.01,0.01}{2}$};
        % Text Node
        \draw (278.36,148.09) node [anchor=north west][inner sep=0.75pt]    {$\textcolor[rgb]{0.82,0.01,0.01}{3}$};
        % Text Node
        \draw (307.29,157.1) node [anchor=north west][inner sep=0.75pt]    {$\textcolor[rgb]{0.82,0.01,0.01}{4}$};
        % Text Node
        \draw (330.97,157.47) node [anchor=north west][inner sep=0.75pt]  [font=\normalsize]  {$\textcolor[rgb]{0.82,0.01,0.01}{5}$};
        % Text Node
        \draw (365.29,161.8) node [anchor=north west][inner sep=0.75pt]  [font=\normalsize]  {$\textcolor[rgb]{0.82,0.01,0.01}{6}$};

    \end{tikzpicture}

    \caption{}
    \label{fig:computation}
\end{figure}

\begin{table}[H]
    \centering
        
    \begin{tabular}{|p{0.02\textwidth}|p{0.29\textwidth}|p{0.29\textwidth}|p{0.29\textwidth}|}
        \hline 
        & \begin{center}
            $\displaystyle 1$
        \end{center}
        & \begin{center}
            $\displaystyle 2$
        \end{center}
        & \begin{center}
            $\displaystyle 3$
        \end{center}
        \\
        \hline 
        A' & 

        \tikzset{every picture/.style={line width=0.75pt}} %set default line width to 0.75pt        
        
        \begin{tikzpicture}[x=0.75pt,y=0.75pt,yscale=-1,xscale=1]
        %uncomment if require: \path (0,80); %set diagram left start at 0, and has height of 80

            %Shape: Circle [id:dp9426623286342491] 
            \draw   (140.5,39.25) .. controls (140.5,36.63) and (142.63,34.5) .. (145.25,34.5) .. controls (147.87,34.5) and (150,36.63) .. (150,39.25) .. controls (150,41.87) and (147.87,44) .. (145.25,44) .. controls (142.63,44) and (140.5,41.87) .. (140.5,39.25) -- cycle ;
            %Shape: Circle [id:dp9409447774071522] 
            \draw  [fill={rgb, 255:red, 255; green, 255; blue, 255 }  ,fill opacity=1 ] (19.5,38.75) .. controls (19.5,36.13) and (21.63,34) .. (24.25,34) .. controls (26.87,34) and (29,36.13) .. (29,38.75) .. controls (29,41.37) and (26.87,43.5) .. (24.25,43.5) .. controls (21.63,43.5) and (19.5,41.37) .. (19.5,38.75) -- cycle ;
            %Shape: Circle [id:dp28373340194468843] 
            \draw   (53.5,39.25) .. controls (53.5,36.63) and (55.63,34.5) .. (58.25,34.5) .. controls (60.87,34.5) and (63,36.63) .. (63,39.25) .. controls (63,41.87) and (60.87,44) .. (58.25,44) .. controls (55.63,44) and (53.5,41.87) .. (53.5,39.25) -- cycle ;
            %Shape: Circle [id:dp010943261234274804] 
            \draw  [fill={rgb, 255:red, 0; green, 0; blue, 0 }  ,fill opacity=1 ] (82.5,39.25) .. controls (82.5,36.63) and (84.63,34.5) .. (87.25,34.5) .. controls (89.87,34.5) and (92,36.63) .. (92,39.25) .. controls (92,41.87) and (89.87,44) .. (87.25,44) .. controls (84.63,44) and (82.5,41.87) .. (82.5,39.25) -- cycle ;
            %Shape: Circle [id:dp9130931662797478] 
            \draw  [fill={rgb, 255:red, 0; green, 0; blue, 0 }  ,fill opacity=1 ] (111.5,39.25) .. controls (111.5,36.63) and (113.63,34.5) .. (116.25,34.5) .. controls (118.87,34.5) and (121,36.63) .. (121,39.25) .. controls (121,41.87) and (118.87,44) .. (116.25,44) .. controls (113.63,44) and (111.5,41.87) .. (111.5,39.25) -- cycle ;
            %Straight Lines [id:da6821601715331341] 
            \draw  [dash pattern={on 4.5pt off 4.5pt}]  (6.83,19.5) -- (6.83,57.5) ;
            %Straight Lines [id:da24074357504111266] 
            \draw  [dash pattern={on 4.5pt off 4.5pt}]  (131.83,19) -- (131.83,57) ;
            %Curve Lines [id:da31454333827231873] 
            \draw    (29,38.75) .. controls (35.88,19.5) and (76.88,17.5) .. (82.5,39.25) ;
            %Straight Lines [id:da3010179875409813] 
            \draw    (92,39.25) -- (111.5,39.25) ;
            %Straight Lines [id:da8308130740996593] 
            \draw    (63,39.25) -- (82.5,39.25) ;
            %Straight Lines [id:da818444900659147] 
            \draw    (121,39.25) -- (140.5,39.25) ;

        \end{tikzpicture}
        & 

        \tikzset{every picture/.style={line width=0.75pt}} %set default line width to 0.75pt        
        
        \begin{tikzpicture}[x=0.75pt,y=0.75pt,yscale=-1,xscale=1]
        %uncomment if require: \path (0,80); %set diagram left start at 0, and has height of 80

            %Shape: Circle [id:dp21011024444394288] 
            \draw   (140.5,39.25) .. controls (140.5,36.63) and (142.63,34.5) .. (145.25,34.5) .. controls (147.87,34.5) and (150,36.63) .. (150,39.25) .. controls (150,41.87) and (147.87,44) .. (145.25,44) .. controls (142.63,44) and (140.5,41.87) .. (140.5,39.25) -- cycle ;
            %Shape: Circle [id:dp45657998792002563] 
            \draw  [fill={rgb, 255:red, 255; green, 255; blue, 255 }  ,fill opacity=1 ] (19.5,38.75) .. controls (19.5,36.13) and (21.63,34) .. (24.25,34) .. controls (26.87,34) and (29,36.13) .. (29,38.75) .. controls (29,41.37) and (26.87,43.5) .. (24.25,43.5) .. controls (21.63,43.5) and (19.5,41.37) .. (19.5,38.75) -- cycle ;
            %Shape: Circle [id:dp27414489119446295] 
            \draw   (53.5,39.25) .. controls (53.5,36.63) and (55.63,34.5) .. (58.25,34.5) .. controls (60.87,34.5) and (63,36.63) .. (63,39.25) .. controls (63,41.87) and (60.87,44) .. (58.25,44) .. controls (55.63,44) and (53.5,41.87) .. (53.5,39.25) -- cycle ;
            %Shape: Circle [id:dp9531070893626458] 
            \draw  [fill={rgb, 255:red, 0; green, 0; blue, 0 }  ,fill opacity=1 ] (82.5,39.25) .. controls (82.5,36.63) and (84.63,34.5) .. (87.25,34.5) .. controls (89.87,34.5) and (92,36.63) .. (92,39.25) .. controls (92,41.87) and (89.87,44) .. (87.25,44) .. controls (84.63,44) and (82.5,41.87) .. (82.5,39.25) -- cycle ;
            %Shape: Circle [id:dp47792027123829517] 
            \draw  [fill={rgb, 255:red, 0; green, 0; blue, 0 }  ,fill opacity=1 ] (111.5,39.25) .. controls (111.5,36.63) and (113.63,34.5) .. (116.25,34.5) .. controls (118.87,34.5) and (121,36.63) .. (121,39.25) .. controls (121,41.87) and (118.87,44) .. (116.25,44) .. controls (113.63,44) and (111.5,41.87) .. (111.5,39.25) -- cycle ;
            %Straight Lines [id:da8378993715868889] 
            \draw  [dash pattern={on 4.5pt off 4.5pt}]  (6.83,19.5) -- (6.83,57.5) ;
            %Straight Lines [id:da45195601853513356] 
            \draw  [dash pattern={on 4.5pt off 4.5pt}]  (131.83,19) -- (131.83,57) ;
            %Curve Lines [id:da7365768075177336] 
            \draw    (29,38.75) .. controls (42.84,16.5) and (103.84,17.5) .. (111.5,39.25) ;
            %Straight Lines [id:da8119009737783229] 
            \draw    (92,39.25) -- (111.5,39.25) ;
            %Straight Lines [id:da5171064641434688] 
            \draw    (121,39.25) -- (140.5,39.25) ;
            %Curve Lines [id:da6311753002974606] 
            \draw    (63,39.25) .. controls (66.84,63.5) and (105.84,63.5) .. (111.5,39.25) ;

        \end{tikzpicture}
        & 

        \tikzset{every picture/.style={line width=0.75pt}} %set default line width to 0.75pt        
        
        \begin{tikzpicture}[x=0.75pt,y=0.75pt,yscale=-1,xscale=1]
        %uncomment if require: \path (0,80); %set diagram left start at 0, and has height of 80

            %Shape: Circle [id:dp21139696063800162] 
            \draw   (140.5,39.25) .. controls (140.5,36.63) and (142.63,34.5) .. (145.25,34.5) .. controls (147.87,34.5) and (150,36.63) .. (150,39.25) .. controls (150,41.87) and (147.87,44) .. (145.25,44) .. controls (142.63,44) and (140.5,41.87) .. (140.5,39.25) -- cycle ;
            %Shape: Circle [id:dp5809181959009715] 
            \draw  [fill={rgb, 255:red, 0; green, 0; blue, 0 }  ,fill opacity=1 ] (19.5,38.75) .. controls (19.5,36.13) and (21.63,34) .. (24.25,34) .. controls (26.87,34) and (29,36.13) .. (29,38.75) .. controls (29,41.37) and (26.87,43.5) .. (24.25,43.5) .. controls (21.63,43.5) and (19.5,41.37) .. (19.5,38.75) -- cycle ;
            %Shape: Circle [id:dp9480857030084567] 
            \draw  [fill={rgb, 255:red, 255; green, 255; blue, 255 }  ,fill opacity=1 ] (53.5,39.25) .. controls (53.5,36.63) and (55.63,34.5) .. (58.25,34.5) .. controls (60.87,34.5) and (63,36.63) .. (63,39.25) .. controls (63,41.87) and (60.87,44) .. (58.25,44) .. controls (55.63,44) and (53.5,41.87) .. (53.5,39.25) -- cycle ;
            %Shape: Circle [id:dp11365216742488393] 
            \draw  [fill={rgb, 255:red, 255; green, 255; blue, 255 }  ,fill opacity=1 ] (82.5,39.25) .. controls (82.5,36.63) and (84.63,34.5) .. (87.25,34.5) .. controls (89.87,34.5) and (92,36.63) .. (92,39.25) .. controls (92,41.87) and (89.87,44) .. (87.25,44) .. controls (84.63,44) and (82.5,41.87) .. (82.5,39.25) -- cycle ;
            %Shape: Circle [id:dp32474323947071315] 
            \draw  [fill={rgb, 255:red, 0; green, 0; blue, 0 }  ,fill opacity=1 ] (111.5,39.25) .. controls (111.5,36.63) and (113.63,34.5) .. (116.25,34.5) .. controls (118.87,34.5) and (121,36.63) .. (121,39.25) .. controls (121,41.87) and (118.87,44) .. (116.25,44) .. controls (113.63,44) and (111.5,41.87) .. (111.5,39.25) -- cycle ;
            %Straight Lines [id:da9134260005894232] 
            \draw  [dash pattern={on 4.5pt off 4.5pt}]  (6.83,19.5) -- (6.83,57.5) ;
            %Straight Lines [id:da8611636639133431] 
            \draw  [dash pattern={on 4.5pt off 4.5pt}]  (131.83,19) -- (131.83,57) ;
            %Curve Lines [id:da7355476068036243] 
            \draw    (29,38.75) .. controls (42.84,16.5) and (103.84,17.5) .. (111.5,39.25) ;
            %Straight Lines [id:da2705602820399352] 
            \draw    (92,39.25) -- (111.5,39.25) ;
            %Straight Lines [id:da0012872188154789965] 
            \draw    (121,39.25) -- (140.5,39.25) ;
            %Curve Lines [id:da8313847266884413] 
            \draw    (63,39.25) .. controls (66.84,63.5) and (105.84,63.5) .. (111.5,39.25) ;

        \end{tikzpicture}
        \\
        \hline 
        B' & 

        \tikzset{every picture/.style={line width=0.75pt}} %set default line width to 0.75pt        
        
        \begin{tikzpicture}[x=0.75pt,y=0.75pt,yscale=-1,xscale=1]
        %uncomment if require: \path (0,80); %set diagram left start at 0, and has height of 80

            %Shape: Circle [id:dp41304541885202584] 
            \draw   (140.5,39.25) .. controls (140.5,36.63) and (142.63,34.5) .. (145.25,34.5) .. controls (147.87,34.5) and (150,36.63) .. (150,39.25) .. controls (150,41.87) and (147.87,44) .. (145.25,44) .. controls (142.63,44) and (140.5,41.87) .. (140.5,39.25) -- cycle ;
            %Shape: Circle [id:dp2795623075522071] 
            \draw  [fill={rgb, 255:red, 255; green, 255; blue, 255 }  ,fill opacity=1 ] (19.5,38.75) .. controls (19.5,36.13) and (21.63,34) .. (24.25,34) .. controls (26.87,34) and (29,36.13) .. (29,38.75) .. controls (29,41.37) and (26.87,43.5) .. (24.25,43.5) .. controls (21.63,43.5) and (19.5,41.37) .. (19.5,38.75) -- cycle ;
            %Shape: Circle [id:dp8017877207273599] 
            \draw   (53.5,39.25) .. controls (53.5,36.63) and (55.63,34.5) .. (58.25,34.5) .. controls (60.87,34.5) and (63,36.63) .. (63,39.25) .. controls (63,41.87) and (60.87,44) .. (58.25,44) .. controls (55.63,44) and (53.5,41.87) .. (53.5,39.25) -- cycle ;
            %Shape: Circle [id:dp920430705623498] 
            \draw  [fill={rgb, 255:red, 0; green, 0; blue, 0 }  ,fill opacity=1 ] (82.5,39.25) .. controls (82.5,36.63) and (84.63,34.5) .. (87.25,34.5) .. controls (89.87,34.5) and (92,36.63) .. (92,39.25) .. controls (92,41.87) and (89.87,44) .. (87.25,44) .. controls (84.63,44) and (82.5,41.87) .. (82.5,39.25) -- cycle ;
            %Shape: Circle [id:dp037995598909518424] 
            \draw  [fill={rgb, 255:red, 0; green, 0; blue, 0 }  ,fill opacity=1 ] (111.5,39.25) .. controls (111.5,36.63) and (113.63,34.5) .. (116.25,34.5) .. controls (118.87,34.5) and (121,36.63) .. (121,39.25) .. controls (121,41.87) and (118.87,44) .. (116.25,44) .. controls (113.63,44) and (111.5,41.87) .. (111.5,39.25) -- cycle ;
            %Straight Lines [id:da6306112126703814] 
            \draw  [dash pattern={on 4.5pt off 4.5pt}]  (6.83,19.5) -- (6.83,57.5) ;
            %Straight Lines [id:da938534181838308] 
            \draw  [dash pattern={on 4.5pt off 4.5pt}]  (131.83,19) -- (131.83,57) ;
            %Curve Lines [id:da3284410262583931] 
            \draw    (29,38.75) .. controls (42.84,16.5) and (74.84,17.5) .. (82.5,39.25) ;
            %Straight Lines [id:da3789488996079132] 
            \draw    (92,39.25) -- (111.5,39.25) ;
            %Straight Lines [id:da739196244710115] 
            \draw    (121,39.25) -- (140.5,39.25) ;
            %Curve Lines [id:da3635901814221658] 
            \draw    (63,39.25) .. controls (66.84,63.5) and (105.84,63.5) .. (111.5,39.25) ;

        \end{tikzpicture}
        & &  \\
        \hline
    \end{tabular}

    \caption{$1$-real floor diagrams with divergence sequence $r_1-l$ contributing to $G^{1,2}_{(\textbf{d}^r;\textbf{d}^l),\textbf{c},0,1}(x_1,y_1,y_2,0)$ in $\mathbf{C}_2$.}
    \label{table5}
        
\end{table}\vspace*{-\baselineskip}

\begin{table}[H]
    \centering
        
    \begin{tabular}{|p{0.02\textwidth}|p{0.29\textwidth}|p{0.29\textwidth}|p{0.29\textwidth}|}
        \hline 
        & \begin{center}
            $\displaystyle 1$
        \end{center}
        & \begin{center}
            $\displaystyle 2$
        \end{center}
        \\
        \hline 
        C' & 

        \tikzset{every picture/.style={line width=0.75pt}} %set default line width to 0.75pt        
        
        \begin{tikzpicture}[x=0.75pt,y=0.75pt,yscale=-1,xscale=1]
        %uncomment if require: \path (0,80); %set diagram left start at 0, and has height of 80

            %Shape: Circle [id:dp27426503879009667] 
            \draw   (140.5,39.25) .. controls (140.5,36.63) and (142.63,34.5) .. (145.25,34.5) .. controls (147.87,34.5) and (150,36.63) .. (150,39.25) .. controls (150,41.87) and (147.87,44) .. (145.25,44) .. controls (142.63,44) and (140.5,41.87) .. (140.5,39.25) -- cycle ;
            %Shape: Circle [id:dp5851041911035323] 
            \draw  [fill={rgb, 255:red, 255; green, 255; blue, 255 }  ,fill opacity=1 ] (19.5,38.75) .. controls (19.5,36.13) and (21.63,34) .. (24.25,34) .. controls (26.87,34) and (29,36.13) .. (29,38.75) .. controls (29,41.37) and (26.87,43.5) .. (24.25,43.5) .. controls (21.63,43.5) and (19.5,41.37) .. (19.5,38.75) -- cycle ;
            %Shape: Circle [id:dp15234696352675736] 
            \draw   (53.5,39.25) .. controls (53.5,36.63) and (55.63,34.5) .. (58.25,34.5) .. controls (60.87,34.5) and (63,36.63) .. (63,39.25) .. controls (63,41.87) and (60.87,44) .. (58.25,44) .. controls (55.63,44) and (53.5,41.87) .. (53.5,39.25) -- cycle ;
            %Shape: Circle [id:dp03288385022117224] 
            \draw  [fill={rgb, 255:red, 0; green, 0; blue, 0 }  ,fill opacity=1 ] (82.5,39.25) .. controls (82.5,36.63) and (84.63,34.5) .. (87.25,34.5) .. controls (89.87,34.5) and (92,36.63) .. (92,39.25) .. controls (92,41.87) and (89.87,44) .. (87.25,44) .. controls (84.63,44) and (82.5,41.87) .. (82.5,39.25) -- cycle ;
            %Shape: Circle [id:dp9291802359329088] 
            \draw  [fill={rgb, 255:red, 0; green, 0; blue, 0 }  ,fill opacity=1 ] (111.5,39.25) .. controls (111.5,36.63) and (113.63,34.5) .. (116.25,34.5) .. controls (118.87,34.5) and (121,36.63) .. (121,39.25) .. controls (121,41.87) and (118.87,44) .. (116.25,44) .. controls (113.63,44) and (111.5,41.87) .. (111.5,39.25) -- cycle ;
            %Straight Lines [id:da8485307975441885] 
            \draw  [dash pattern={on 4.5pt off 4.5pt}]  (6.83,19.5) -- (6.83,57.5) ;
            %Straight Lines [id:da702557507548169] 
            \draw  [dash pattern={on 4.5pt off 4.5pt}]  (131.83,19) -- (131.83,57) ;
            %Curve Lines [id:da6590088318760626] 
            \draw    (29,38.75) .. controls (42.84,16.5) and (74.84,17.5) .. (82.5,39.25) ;
            %Straight Lines [id:da23319913487102417] 
            \draw    (92,39.25) -- (111.5,39.25) ;
            %Straight Lines [id:da840280564550687] 
            \draw    (121,39.25) -- (140.5,39.25) ;
            %Curve Lines [id:da06978917003213292] 
            \draw    (63,39.25) .. controls (66.84,63.5) and (105.84,63.5) .. (111.5,39.25) ;

        \end{tikzpicture}
        & 

        \tikzset{every picture/.style={line width=0.75pt}} %set default line width to 0.75pt        
        
        \begin{tikzpicture}[x=0.75pt,y=0.75pt,yscale=-1,xscale=1]
        %uncomment if require: \path (0,80); %set diagram left start at 0, and has height of 80

            %Shape: Circle [id:dp9919335809692724] 
            \draw   (140.5,39.25) .. controls (140.5,36.63) and (142.63,34.5) .. (145.25,34.5) .. controls (147.87,34.5) and (150,36.63) .. (150,39.25) .. controls (150,41.87) and (147.87,44) .. (145.25,44) .. controls (142.63,44) and (140.5,41.87) .. (140.5,39.25) -- cycle ;
            %Shape: Circle [id:dp8554069826682298] 
            \draw  [fill={rgb, 255:red, 255; green, 255; blue, 255 }  ,fill opacity=1 ] (19.5,38.75) .. controls (19.5,36.13) and (21.63,34) .. (24.25,34) .. controls (26.87,34) and (29,36.13) .. (29,38.75) .. controls (29,41.37) and (26.87,43.5) .. (24.25,43.5) .. controls (21.63,43.5) and (19.5,41.37) .. (19.5,38.75) -- cycle ;
            %Shape: Circle [id:dp06689658304931057] 
            \draw  [fill={rgb, 255:red, 0; green, 0; blue, 0 }  ,fill opacity=1 ] (48.5,38.75) .. controls (48.5,36.13) and (50.63,34) .. (53.25,34) .. controls (55.87,34) and (58,36.13) .. (58,38.75) .. controls (58,41.37) and (55.87,43.5) .. (53.25,43.5) .. controls (50.63,43.5) and (48.5,41.37) .. (48.5,38.75) -- cycle ;
            %Shape: Circle [id:dp29065543492436907] 
            \draw  [fill={rgb, 255:red, 255; green, 255; blue, 255 }  ,fill opacity=1 ] (82.5,39.25) .. controls (82.5,36.63) and (84.63,34.5) .. (87.25,34.5) .. controls (89.87,34.5) and (92,36.63) .. (92,39.25) .. controls (92,41.87) and (89.87,44) .. (87.25,44) .. controls (84.63,44) and (82.5,41.87) .. (82.5,39.25) -- cycle ;
            %Shape: Circle [id:dp006127876626027429] 
            \draw  [fill={rgb, 255:red, 0; green, 0; blue, 0 }  ,fill opacity=1 ] (111.5,39.25) .. controls (111.5,36.63) and (113.63,34.5) .. (116.25,34.5) .. controls (118.87,34.5) and (121,36.63) .. (121,39.25) .. controls (121,41.87) and (118.87,44) .. (116.25,44) .. controls (113.63,44) and (111.5,41.87) .. (111.5,39.25) -- cycle ;
            %Straight Lines [id:da0034386269262743907] 
            \draw  [dash pattern={on 4.5pt off 4.5pt}]  (6.83,19.5) -- (6.83,57.5) ;
            %Straight Lines [id:da7763519025433899] 
            \draw  [dash pattern={on 4.5pt off 4.5pt}]  (131.83,19) -- (131.83,57) ;
            %Straight Lines [id:da1651762102640456] 
            \draw    (92,39.25) -- (111.5,39.25) ;
            %Straight Lines [id:da2635793667467018] 
            \draw    (121,39.25) -- (140.5,39.25) ;
            %Curve Lines [id:da03067331710749488] 
            \draw    (58,38.75) .. controls (61.84,63) and (105.84,63.5) .. (111.5,39.25) ;
            %Straight Lines [id:da9431168687658705] 
            \draw    (29,38.75) -- (48.5,38.75) ;

        \end{tikzpicture}
        \\
        \hline
    \end{tabular}
    \caption{$1$-real floor diagrams with divergence sequence $r_2-l$ contributing to $G^{1,2}_{(\textbf{d}^r;\textbf{d}^l),\textbf{c},0,1}(x_1,y_1,y_2,0)$ in $\mathbf{C}_2$.}
    \label{table6}
        
\end{table}

\printbibliography

\end{document}